\documentclass[reqno,12pt]{preprint}
\usepackage[marginparwidth=1in]{geometry}
\usepackage[full]{textcomp}
\usepackage[osf]{newtxtext}
\usepackage{cabin}
\usepackage{enumerate}
\usepackage{enumitem}
\usepackage{bbm}
\usepackage{dsfont}

\usepackage[varqu,varl]{inconsolata}
\usepackage[cal=boondoxo]{mathalfa}

\usepackage{comment}
\usepackage{hyperref}
\usepackage{breakurl}
\usepackage{mathrsfs}
\usepackage{booktabs}
\usepackage{caption}

\usepackage{tikz}
\usepackage{amsmath} 
\usepackage{mhequ} 
\usepackage{mhsymb} 
\usepackage{mhenvs} 
\usepackage{microtype}
\usepackage{amssymb} 
\usepackage{eufrak}

\usepackage{wasysym}
\usepackage{centernot}

\usepackage{oldgerm}

\def\CC{\mathscr{C}}

\def\CC{\mathcal{C}}

\def\CD{\mathcal{D}}
\def\CE{\mathcal{E}}
\def\CF{\mathcal{F}}
\def\CG{\mathcal{G}}
\def\CH{\mathcal{H}}

\def\CN{\mathcal{N}}
\def\CO{\mathcal{O}}

\def\CS{\mathcal{S}}

\def\C{{\mathbf C}}

\def\E{{\mathbf E}}

\def\J{{\mathbf J}}

\def\N{{\mathbf N}}

\def\P{\mathbf{P}}

\def\R{\mathbf{R}}

\def\T{\mathbf{T}}

\def\Z{\mathbf{Z}}

\def\i{\bold i}
\def\j{\bold j}
\def\k{\bold k}

\newcommand{\Exp}{\mathbf{E}}

\renewcommand{\R}{\mathbb{R}}
\renewcommand{\C}{\mathbb{C}}
\renewcommand{\N}{\mathbb{N}}

\renewcommand{\T}{\mathbb{T}}
\renewcommand{\Z}{\mathbb{Z}}
\renewcommand{\E}{\mathbb{E}}
\renewcommand{\P}{\mathbb{P}}

\renewcommand{\J}{\mathbb{J}}

\newcommand{\cA}{\mathcal{A}}
\newcommand{\cB}{\mathcal{B}}

\newcommand{\cD}{\mathcal{D}}
\newcommand{\cF}{\mathcal{F}}

\newcommand{\cH}{\mathcal{H}}

\newcommand{\cK}{\mathcal{K}}
\newcommand{\cL}{\mathcal{L}}
\newcommand{\cM}{\mathcal{M}}
\newcommand{\cN}{\mathcal{N}}

\newcommand{\dd}{\mathrm{d}}      

\newcommand{\Id}{\mathrm{Id}}

\newcommand{\1}{\mathds{1}}

\newcommand{\ua}{\uparrow}

\newcommand{\vphi}{\varphi}
\def\sym{{\mathrm{sym}}}

\colorlet{darkblue}{blue!90!black}
\colorlet{darkred}{red!90!black}
\colorlet{darkgreen}{green!50!black}
\colorlet{darkyellow}{yellow!90!black}

\def\one{\mathrm{(I)}}
\def\two{\mathrm{(II)}}
\def\three{\mathrm{(III)}}
\def\four{\mathrm{(IV)}}

\newcommand{\half}{\frac{1}{2}}

\newcommand{\phalf}{^\half}

\def\hot{\otimes^n h}

\def\hotm{\otimes^{n-1} h}
\def\hotmm{\otimes^{n-2} h}

\newcommand{\thec}{c}			
\def\indN#1{{	\J^N_{#1}	}}
\newcommand{\fock}{\Gamma L^2}	

\def\sint{I}					
\def\wc{\CH}					

\def\qvar#1{\langle #1 \rangle}	

\def\swn{{	\eta	}}
\def\filt{{ \CG	}}		

\newcommand{\gen}{\cL^N}
\newcommand{\gens}{\cL_0^N}
\newcommand{\gensy}{\cL_0}
\newcommand{\gena}{\cA^N}
\newcommand{\genap}{\cA^N_+}
\newcommand{\genam}{\cA^N_-}

\newcommand{\gensyf}{\mathfrak{L}_0}
\newcommand{\genapf}{\mathfrak{A}^N_+} 
\newcommand{\genamf}{\mathfrak{A}^N_-} 
\newcommand{\nonlin}{\cK^N}
\newcommand{\inonlin}{\cB^N}
\newcommand{\inonlintilde}{\tilde\cB^N}
\newcommand{\nonlintilde}{\tilde\cK^N}

\newcommand{\const}{\lambda_N}
\newcommand{\donst}{\nu_N}
\def\donsth{\donst\phalf}

\def\energy{\CE^N}
\def\martf{M^N}			
\def\martb{\bar M^N}	
\def\pois{H^N}			
\def\poisl{K^N}			
\newcommand{\martnl}{\cM^N}  		

\def\she{X^N} 	

\usepackage{fancyhdr}
\title{2D Anisotropic KPZ at stationarity: scaling, tightness and non triviality}
\makeatletter
\let\runtitle\@title
\makeatother
\date{%
    \today
    }

\begin{document}

\maketitle

\vspace{-2cm}

\noindent{\large \bf Giuseppe Cannizzaro$^1$,  Dirk Erhard$^2$, Philipp Sch\"onbauer$^3$}
\newline

\noindent{\small $^1$University of Warwick, UK\\%
    $^2$Universidade Federal da Bahia, Brazil\\%
    $^3$Imperial College London, UK\\}

\noindent\email{giuseppe.cannizzaro@warwick.ac.uk, 
erharddirk@gmail.com, \\
p.schoenbauer@imperial.ac.uk}
\newline

\begin{abstract}
In this work we focus on the two-dimensional anisotropic KPZ (aKPZ) equation, which is formally given by  
\begin{equation*}
\partial_t h =\frac{\nu}{2}\Delta h + \lambda((\partial_1 h)^2 - (\partial_2 h)^2) + \nu^\half \xi,
\end{equation*}
where $\xi$ denotes a noise which is white in both space and time, and $\lambda$ and $\nu$ are positive constants. 
Due to the wild oscillations of the noise and the quadratic 
nonlinearity, the previous equation is classically ill-posed. It is not possible
to linearise it via the Cole-Hopf transformation and the pathwise techniques for singular SPDEs 
(the theory of Regularity Structures by M. Hairer or the paracontrolled distributions approach of M. Gubinelli, P. Imkeller, 
N. Perkowski) are not applicable. 
In the present work, we consider a regularised version of aKPZ which preserves its invariant measure.  We 
show that in order to have subsequential limits once the regularisation is removed, 
it is necessary to suitably renormalise $\lambda$ and $\nu$. 
Moreover, we prove that, in the regime suggested by the (non-rigorous) renormalisation group computations
of [D.E. Wolf, ``Kinetic roughening of vicinal surfaces'', Phys. Rev. Lett., 1991], i.e. $\nu$ constant and the coupling 
constant $\lambda$ converging to $0$ as the inverse of the square root logarithm, 
any limit differs from the solution to the linear equation obtained by simply dropping the nonlinearity in aKPZ.   
%
%
%
%
%
%
%
%
%
%
%
%
%
%
\end{abstract}


\tableofcontents

\section{Introduction}

The KPZ equation is a (singular) stochastic partial differential equation (SPDE), whose formal expression is 
\begin{equation}\label{eq:anKPZ}
\partial_t h= \nu \Delta h + \langle \nabla h, Q\nabla h\rangle + \sqrt{D}\xi,
\end{equation}
where $\xi$ is a space-time white noise in spatial dimension $d$, $Q$ is a $d\times d$-matrix, and $\nu$ and 
$D$ are positive constants. The importance of this equation stems from the fact that it encodes (via $Q$, $\nu$ and $\lambda$) 
the universal features of randomly evolving surfaces and it is supposed to arise as the limit of a large class of properly 
rescaled particle systems. The difficulty in establishing its universality is already on the level of the equation since, from an 
analytic viewpoint, it is ill-posed in any dimension. This is due to the fact that the noise $\xi$ is too irregular for the  
non-linear term to be canonically defined. 

The only dimension in which a rigorous solution theory has been established (for any value of the constants $\nu,\,Q$ and $D$)
and the universality claim corroborated, is $d=1$. There are by now different approaches that lead to well-posedness: 
the Cole-Hopf transformation that turns~\eqref{eq:anKPZ} into the {\it linear} multiplicative stochastic heat equation~\cite{BG}; 
the martingale approach which leads to the notion of energy solution~\cite{GubinelliJara2012,GPuni}; pathwise techniques, 
namely rough paths~\cite{KPZ}, regularity structures~\cite{Hai} and paracontrolled calculus~\cite{Para,KPZreloaded}. 
In particular, the theory of regularity structures and paracontrolled calculus additionally apply to a much larger class of equations
and, from their introduction, the field of (singular) SPDEs has experienced a tremendous growth. 
That said, their applicability is restricted to those equations that are {\it subcritical} which heuristically means that, at small 
scales, the nonlinearity does not matter much and the solution behaves (regularity wise) as the linear part of the equation. 
For~\eqref{eq:anKPZ}, this is the case only for $d=1$, while in $d=2$ and $d\geq3$ (which are said to be the {\it critical} and 
{\it supercritical} regimes respectively) the pathwise approaches break down. 

Only recently the critical and supercritical regimes started to be investigated. In the latter case 
physicists (see~\cite{Kardar}) predict that, for the parameters $\nu,\,Q$ and $D$ in a suitable window, 
the non-linearity should not matter much at large scales, so that, 
taking a smooth noise, rescaling the height function $h$ according to $h^\eps(t,x)\eqdef \eps^{\frac{d}{2}}h(t/\eps^2,x/\eps)$
and subtracting the average growth, the fluctuations should be the same as those of the solution of the linear stochastic 
heat equation. 
Partial results in this direction have been established in the case $Q=\lambda\, \Id_d$, for $\Id_d$
being the $d\times d$ identity matrix and the coupling constant $\lambda>0$ sufficiently small, first by~\cite{Magnen} 
via renormalisation group techniques and later 
by~\cite{DGRZ,CCM1,CCM2} (see also~\cite{Gu2018b} for the case of the multiplicative stochastic heat equation)
\footnote{In the supercritical regime, a phase transition is expected depending on $\lambda$ but the exact value at which 
the transition happens is still unknown.}.

The critical case, $d=2$, shows an even deeper structure. Indeed, already from the physics perspective
this regime is more delicate since finer details of the equation, and in particular the sign of $\det Q$, 
might influence its large scale dynamics. The importance of the matrix $Q$ can be understood from a microscopic 
viewpoint. Indeed, heuristically speaking, it is expected that the macroscopic 
{\it average behaviour} of a microscopic surface is given by the solution of a PDE of the form 
\begin{equation}\label{e:hydro}
\partial_t u=v(\nabla u)
\end{equation}
where $v$ is a deterministic scalar valued map depending on the specific (microscopic) features of the model at hand. 
Now, since~\eqref{eq:anKPZ} should represent the (universal) fluctuations of the surface {\it around} its hydrodynamic limit,
a second order expansion of~\eqref{e:hydro} leads to the identification of $Q$ with the Hessian of $v$. 
Through (non-rigorous) renormalisation group techniques, Wolf showed in~\cite{W91} that~\eqref{eq:anKPZ} gives rise 
to {\it two different} universal behaviours depending on the sign of $\det Q$. 
If $\det Q > 0$, the so called {\it isotropic KPZ class}, then the fluctuations should grow in time as $t^{\beta}$ for some $\beta > 0$, 
and the spatial correlation should grow as the distance to the power $\frac{2\beta}{(\beta + 1)}$, see~\cite{Kardar} 
while for $\det Q \leq 0$, the {\it anisotropic KPZ class}, the non-linearity should morally play no role and the behaviour 
should be the same as the solution to the stochastic heat equation in dimension 2. Note that the latter in particular means 
that the value of $\beta$ mentioned above should be equal to zero, and the correlations explode logarithmically. 
We emphasise that it is nowhere stated that anisotropic KPZ equation
\emph{coincides} with the stochastic heat equation, only the correlations should be of the same order.
This is though expected, especially in view of the works~\cite{BCF,BCT}, where the scaling limit of the models there 
considered is obtained via a limit transition, namely a first limit reduces the models to a system of linear SDEs and, thanks to a 
second limit, the linear stochastic heat equation is derived. 

Numerically, the conjecture for the isotropic case was for instance confirmed in~\cite{Tang} for two specific models 
where it turned out to be the case that $\beta \approx 0.24$, while the anisotropic was studied in~\cite{Healy}.

Mathematically an even deeper structure has been found for $\det Q>0$. Indeed, upon choosing $Q=\lambda \Id_2$, 
and $\lambda\sim \sqrt{\hat\lambda/\log N}$, where $N$ is a regularisation parameter, 
the work of Caravenna, Sun and Zygouras~\cite{CSZ1} shows that there is a phase transition (for the one point distribution) at 
$\hat\lambda=2\pi$. Later in~\cite{CD19}, for $\hat\lambda>0$ sufficiently close to $0$, 
it was shown that a sequence of 
approximations of~\eqref{eq:anKPZ} is tight, result then improved in~\cite{CSZ}, where not only tightness, 
but also uniqueness and characterisation of the limit was obtained in the whole
interval $\hat\lambda\in(0,2\pi)$. They proved that the limit is given by the solution of a stochastic heat equation, different from
the one obtained by simply dropping the nonlinear term in~\eqref{eq:anKPZ} (see also~\cite{Gu}). 
\newline

In the present paper, we will focus on the anisotropic KPZ class. For numerous (discrete) models the Hessian of $v$ 
appearing in~\eqref{e:hydro} has been computed, see for example~\cite{Borodin2014, Toninelli2017, BT}, and its determinant 
proven to be negative. Precise results were obtained concerning the hydrodynamic behaviour and the convergence 
of the invariant measure to the Gaussian free field (see~\cite{Borodin2014, Legras2019}).
What hinders still the progress is that the statements mentioned so far on the fluctuations have been established at {\it fixed} time 
and it is not clear how one can show that the time fluctuations are really of the logarithmic order 
as expected (some advances have been
made in~\cite{Toninelli2017, CT} where a $\log t$ upper bound has been obtained for the time increment). 

To shed some light on the behaviour as a \emph{process} for a model belonging to the anisoptropic KPZ class, 
we will be working directly at the level of the equation~\eqref{eq:anKPZ}. We make a specific choice of the matrix $Q$, 
i.e. $Q=\text{diag}(1,-1)$, and of initial condition, i.e. we start from the invariant measure, that with this choice of $Q$
can be shown to exist (see Lemma~\ref{lem:invariant} below). 
The aforementioned paper of Wolf suggests that in order to see the universal fluctuations
it is necessary to renormalise the coupling constants. Therefore, we were lead to study the following family of 
approximations
\begin{align}\label{e:kpz:reg}
\partial_t h^N = \frac{\donst}{2} \Delta h^N
+
\const \Pi_N \Big((\Pi_N \partial_1 h^N)^2 - (\Pi_N \partial_2 h^N)^2\Big) + \donst^{\frac{1}{2}}\xi\,,\qquad h^N_0=\tilde\eta
\end{align}
in which
\begin{itemize}[noitemsep]
\item[-] $\tilde\eta$ is a Gaussian free field on $\T^2$, i.e. a Gaussian field whose covariance function is
\begin{equ}
\E[\tilde\eta(\varphi)\tilde\eta(\psi)]=\langle(\Delta)^{-1}\varphi,\psi\rangle_{L^2(\T^2)} \,,\qquad\text{for all $\varphi,\psi\in H^{-1}(\T^2)$,}
\end{equ}
and it is assumed that the $0$ Fourier mode of $\varphi$ and $\psi$ is $0$.  
\item[-] $\xi$ is a space time white noise on $\R_+\times\T^2$ independent of $\tilde\eta$, 
i.e. a Gaussian field whose covariance function is 
\begin{equ}
\Exp[\xi(\varphi)\xi(\psi)]=\left\langle\varphi-\int_{\T^2}\varphi(x)\dd x,\psi-\int_{\T^2}\psi(x)\dd x\right\rangle_{L^2(\R_+\times\T^2)}
\end{equ}
for all $\varphi,\psi\in L^2(\R_+\times\T^2)$, 
\item[-] $\Pi_N$ is the operator acting in Fourier space by cutting the modes
higher that $N$, i.e. 
\begin{equ}
(\Pi_N w)_k\eqdef w_k \mathds{1}_{|k|_\infty \le N}
\end{equ} 
and $w_k$ is the $k$-th Fourier component of $w$,
\item[-] $\donst$ and $\const$ are positive constants allowed to depend on the regularisation parameter $N$. 
\end{itemize} 
In Theorem~\ref{thm:tight}, which is a consequence of Theorem \ref{thm:tightness} and 
Theorem \ref{thm:tightnessAKPZ} below, we identify a {\it family} of different scalings for $\const$ and $\donst$ for which 
the sequence $h^N$ admits subsequential limits in Besov-H\"older spaces of suitable regularity (see~\eqref{def:Besov}
for a precise definition of these spaces).  

\begin{theorem}\label{thm:tight}
Let $N\in \N$ and consider the sequence of equations in~\eqref{e:kpz:reg} 
started from the invariant measure, given by the Gaussian free field $h^N(0)=\tilde\eta$.
Then, provided that
\begin{equation}\label{e:ScalingRCintro}
\const\donst^{-\half}\sim \sqrt{\frac{ 1}{\log N}}\,,\qquad \text{as $N$ tends to $\infty$,}
\end{equation}
the sequence $\{h^N\}_N$ is tight in $C_T^\gamma\CC^\alpha$ for any $\gamma<1/2$ and $\alpha<-1$. Moreover, if 
$\donst=1$ for all $N\in\N$, then tightness holds for any $\alpha<0$ and $\gamma=0$. 
\end{theorem}

Let us point out some aspects of the previous theorem, which mark the difference from the results mentioned above on 
critical SPDEs. Notice that, for the equation we are considering, there is {\it no Cole-Hopf transform} which could
turn~\eqref{e:kpz:reg} into a linear SPDE and therefore no explicit representation of the solution is available. In other words,
we are forced to work directly with the equation itself and make sense of its nonlinearity. 
Moreover (at least in the case $\donst=1$ and $\const$ satisfies~\eqref{e:ScalingRCintro}), we obtain tightness for the 
sequence in the space with {\it optimal} regularity. This can be seen by power counting since $\xi$ has regularity at most $-2$
and the regularising effect of the Laplacian gains $2$. 
At last, notice that, according to~\eqref{e:ScalingRCintro}, we are allowed to take $\const=\donst=(\log N)^{-1}$. 
This is interesting since, by the scaling properties
of $\xi$, it corresponds to the situation in which one starts from the original equation ($\const=\donst=1$) and 
looks at times of order $(\log N)^{-1}$, i.e. $h^N(t,x)\eqdef h(t/\log N,x)$. In other words, we do not modify the equation
but identify the time scale at which we (should) see the relevant behaviour. 
\newline

That said, the previous statement does not rule out the possibility that the limit is trivial, i.e. it is simply constant in time
or reduces to the solution of an equation in which the summands containing a vanishing factor disappear, which would mean 
that the strength at which they converge to $0$ is too strong. 

Upon choosing $\donst=1$, we are indeed able to show that any limit point has finite non-zero energy, which 
in particular implies that 
it is not trivial. Here, we say that a stochastic process $\{Y_t\}_{t\in [0,T]}$ has {\it finite energy} if 
\begin{equation}\label{eq:energy}
\sup_{\pi=\{t_i\}_i}\E\left[\sum_{i} (Y_{t_{i+1}}-Y_{t_i})^2\right]<\infty
\end{equation} 
where the supremum is over all the partitions $\pi$ of $[0,T]$. 

\begin{theorem}\label{thm:nontriviality}
In the setting of Theorem~\ref{thm:tight} assume that $\nu_N=1$. Then, for any test function $\varphi$, any limit point of
the sequence 
\begin{equ}
\left\{\int_0^t\const \Pi_N \Big((\Pi_N \partial_1 h^N)^2 - (\Pi_N \partial_2 h^N)^2\Big)(s,\varphi)\dd s\right\}_{t}
\end{equ}
is a process with finite non-zero energy. 
\end{theorem}

Let us remark that in the paper of Wolf, the scaling chosen in the previous statement is indeed the relevant regime according 
to his renormalisation group computations, see~\cite[eq. (10)-(11)-(12)]{W91}\footnote{Indeed, the equations mentioned 
seem to suggest that, for $Q$ as in our case, in order to get the effective constants, one should let the 
strength of the noise to $0$. By scaling properties of the equation, this is equivalent to taking the nonlinearity to $0$.}. 

Theorem~\ref{thm:nontriviality} is proved in Proposition~\ref{p:LimitPoints} and Theorem~\ref{thm:AKPZ-nontriviality}, where it 
is actually shown more. In particular, our results suggest that {\it any} subsequential limit of $\{h^N\}_N$ will contain a 
{\it new} noise which is 
produced by the dynamics itself. Understanding the nature of this new noise (and its relation to the original one) will be crucial 
in the characterisation of the limit points and is
currently being investigated by the authors.

\subsection{Strategy}

Using tools from Malliavin calculus, we show in Lemma~\ref{lem:invariant} that the invariant measure of 
$h^N$  is given by a Gaussian free field $\tilde\eta$. 
Starting from the invariant measure, we use ideas from \cite{GubinelliJara2012} (established in the study of 
energy solutions in the one-dimensional case) to show that in the scaling regime 
(\ref{e:ScalingRC}) the sequence of solutions is tight, see Theorem \ref{thm:tightness}.
The crucial observation \eqref{e:solPoiNL} is that there exists an explicit functional of $h^N$, called $\pois$, with the property 
that the non-linearity at $h^N$ equals $\gens\pois$, where $\gens$ denotes the generator of the underlying 
linear equation \eqref{e:gens}. 
Using martingale techniques, we are able to obtain bounds which are strong enough to control the non-linearity and 
to establish tightness of the sequence of solutions (see Lemmas~\ref{lem:ito-trick} and~\ref{lem:energy-estimate}). 

We rule out triviality by establishing a non-vanishing lower bound on the second moment of the integral in time of the non linearity, 
see Corollary~\ref{cor:NontrivNonlin}.
Inspired by the analysis of the generator for the one dimensional KPZ equation in~\cite{GPGen} and of the diffusion
coefficient for the asymmetric simple exclusion process in $d=1,2$ of~\cite{Landim2004}, we show that its Laplace transform 
is non-zero in the limit as $N$ tends to infinity. 
The main tool we use for this is the variational formula presented in Lemma~\ref{lem:VarFor}.

\begin{remark}
We want to stress that, in principle, the techniques we adopt are sufficiently flexible 
to be used for other equations at criticality for which 
the invariant measure is explicitly known (e.g. the equations in~\cite[Sections 6 and 7]{GubinelliJara2012}). Moreover, 
since they were inspired by tools introduced in the particle systems context, we think that our approach might prove useful in 
establishing existence of subsequential limits for particle systems and
improve our understanding of their large scale behaviour (e.g. the time evolution).  
\end{remark}

\subsection{Structure of the article}

In Section \ref{S:Malliavin} we recall basic facts from Malliavin calculus, which we use in Section \ref{S:Properties} to show that 
the Gaussian free field is indeed invariant for $h^N$ and to analyse the generator of the Markov process 
$\{h^N(t)\}_t$. In Section \ref{sec:Tightness} we then establish tightness of $h^N$ and prove Theorem~\ref{thm:tight}. 
In Section \ref{S:nontrivialty}, we show non-triviality of the non-linearity and prove Theorem~\ref{thm:nontriviality}. 
We conclude the paper with Section~\ref{sec:Consequence}, in which we explore further consequences of the bounds established 
in Section~\ref{sec:Tightness}.

\subsection*{Notations and function spaces}
The notation $\Z^2_0$ always refers to $\Z^2\setminus\{0\}$ and 
$\T^2$ denotes the two-dimensional torus of side length $2\pi$. 
We equip the space $L^2(\T^2;\C)$ with the Fourier basis $\{e_k\}_{k\in\Z^2}$ defined via 
$e_k(x) \eqdef \frac{1}{2\pi} e^{i k \cdot x}$. 
The basis functions $e_k$ can be decomposed in their real and imaginary part, so that $e_k=a_k+\iota b_k$ and 
the system $\{\sqrt{2}a_k\}_{k \in \Z^2} \sqcup \{\sqrt{2} b_k\}_{k \in \Z^2_0}$ forms a real valued orthonormal basis of $L^2(\T^2)$. 
The Fourier transform, denoted by $\cF$ and at times also by $\hat\cdot$, is given by the formula
	\begin{equation}\label{e:FT}
	\cF(\varphi)(k) \eqdef \varphi_k = \int_{\T^2} \varphi(x) e_{-k}(x)\dd x\,.
	\end{equation}
For any real valued distribution $\eta\in\cD'(\T^2)$ and $k\in\Z^2$, 
its Fourier transform is given by the (complex) pairing 
\begin{equation}\label{e:complexPairing}
\eta_k\eqdef \eta(e_{-k})=\eta(a_k)-\iota \eta(b_k)
\end{equation} 
so that $\overline{\eta(e_k)}=\eta(e_{-k})$. 
Moreover, we recall that the Laplacian $\Delta$ on $\T^2$ has eigenfunctions $\{e_k\}_{k \in \Z^2}$ with eigenvalues $\{-|k|^2\,:\,k\in\Z^2\}$, and we define the operator $(-\Delta)^\theta$
	by its action on the basis elements 
	\begin{equation}\label{e:fLapla}
	(-\Delta)^\theta e_k\eqdef |k|^{2\theta}e_k(x).
	\end{equation}	
	
We will work mostly in Besov spaces. For a thorough exposition on these spaces and their properties, 
we refer the interested reader to~\cite{BCD}, see also~\cite[App.~A]{Para} for a review of the results which we will need below. 
Besov spaces are defined via a dyadic partition of unity $(\chi,\varrho)\in\cD$, i.e. $\chi$
and $\varrho$ are non-negative radial functions such that 
\begin{itemize}[label=-,noitemsep]
\item the supports of $\chi$ and $\varrho$ are respectively contained in a ball and an annulus, 
\item $\chi(x)+\sum_{j\geq 0}\varrho(2^{-j} x)=1$ for all $x\in\R^d$,
\item $\supp(\chi)\cap\supp\varrho(2^{-j}\cdot)=\emptyset$ for all $j\geq 1$ and 
	$\supp(\varrho(2^{-j}\cdot)\cap\supp\varrho(2^{-i}\cdot))=\emptyset$ whenever $|i-j|> 1$. 
\end{itemize}
For any distribution $u\in\cD'(\T^2)$, the Littlewood-Paley blocks are defined as
\begin{equ}
\Delta_{-1} u=\cF^{-1}(\chi\cF(u))\,,\qquad\text{and}\qquad \Delta_{j} u=\cF^{-1}(\varrho_j\cF(u))\,,\quad j\geq 1
\end{equ}
where $\varrho_j\eqdef\varrho(2^{-j\cdot})$. Since $K_j\eqdef \cF^{-1}\varrho_j$ is a smooth function, so is $\Delta_ju=K_j\ast u$.
Given $\alpha\in\R$, $p,q\in[1,+\infty)$, the Besov space $B^\alpha_{p,q}$ is given by 
\begin{equation}\label{def:Besov}
B^\alpha_{p,q}(\T^2)\eqdef\Big\{u\in\cD'(\T^2)\,:\,\|u\|_{B^\alpha_{p,q}}^q\eqdef
\sum_{j\geq -1} 2^{\alpha j q}\|\Delta_j u\|_{L^p(\T^2)}^q<\infty\Big\}\,.
\end{equation}
In the special case $p=q=\infty$, the norm is 
\begin{equ}
\|u\|_{B^\alpha_{\infty,\infty}}\eqdef \sup_{j\geq -1} 2^{\alpha j}\|\Delta_j u\|_{L^\infty(\T^2)}
\end{equ}
and, since this is the space we will mainly work with, we set $\CC^\alpha\eqdef B^\alpha_{\infty,\infty}$ and denote 
the corresponding norm by $\|u\|_\alpha\eqdef\|u\|_{B^\alpha_{\infty,\infty}}$. 
This notation is justified by the fact that for $\alpha>0$, $\alpha \notin\N$ the space $B^\alpha_{\infty,\infty}$ 
coincides with the usual space of $\alpha$-H\"older continuous functions. 
We also point out that for $p=q=2$ and $\alpha\in\R$, $B^\alpha_{2,2}=H^\alpha$, where the latter is the usual 
Sobolev space of regularity index $\alpha$, whose norm (on the torus) can be written as 
\begin{equ}
\|u\|^2_{\alpha,2}=\|u\|^2_{H^\alpha}=\sum_{k\in\Z^2}(1+|k|^2)^\alpha|u_k|^2.
\end{equ}
Restricted to the subspace of distribution $u$ with $u_0=0$, one may replace $1+|k|^2$ by $|k|^2$. 

We will need to following classical embedding theorem for Besov spaces  
(see e.g.~\cite[Lemma~A.2]{Para}). 
\begin{lemma}
For any $\alpha\in\R$, $1\leq p_1\leq p_2\leq\infty$ and $1\leq q_1\leq q_2\leq \infty$
one has
\begin{equation}\label{e:BesovEmb}
	B^\alpha_{p_1,q_1}(\T^2)\hookrightarrow B^{\alpha-2\left(\frac{1}{p_1}-\frac{1}{p_2}\right)}_{p_2,q_2}(\T^2).
\end{equation}
In particular one has $\| u\|_{\alpha-2/p}\leq \|u\|_{B^\alpha_{p,p}}$. 
\end{lemma}

At last, we will denote the space of $\gamma$-H\"older continuous functions on $[0,T]$ with values in a Banach space $B$
by $C_T^\gamma B$.

\section*{Acknowledgements}
We are grateful to Martin Hairer, Milton Jara, Nicolas Perkowski, Fabio Toninelli, and Nikolaos Zygouras for 
helpful discussions and suggestions. A special thanks goes to Nicolas Perkowski for many useful comments and for
pointing out a mistake in an earlier version of the paper. We also thank Ivan Corwin for mentioning
the references~\cite{BCF,BCT}. 
G. C. gratefully acknowledges financial support via the EPSRC grant EP/S012524/1. 
D. E. gratefully acknowledges financial support 
from the National Council for Scientific and Technological Development - CNPq via a 
Universal grant 409259/2018-7. P. S.  acknowledges funding through Martin Hairer's ERC consolidator grant, project 615897.

\section{A primer on Wiener space analysis and Malliavin calculus}
\label{S:Malliavin}

We recall basic tools from Malliavin calculus which we will use below. Most of this is taken from~\cite[Chapter 1]{Nualart2006} to which we refer the interested reader (see also~\cite{GPnotes,GPGen}).

Let $(\Omega, \cF,\P)$ be a complete probability space and $H$ a real separable Hilbert space, with scalar product 
$\langle\cdot,\cdot\rangle$. A stochastic process $\{\eta(h)\,:\,h\in H\}$ is called {\it isonormal Gaussian process} 
if $\eta$ is a family of centred jointly Gaussian random variables whose correlations are given by 
$\E[\eta(h)\eta(g)]=\langle h,g\rangle$. 
Given an isonormal Gaussian process $\eta$ on $H$ and $n\in\N$, we define $\cH_n$ as the closed linear subspace of 
$L^2(\eta)\eqdef L^2(\Omega)$ generated by the random variables $H_n(\eta(h))$, 
where $H_n$ is the $n$-th Hermite polynomial, and 
$h\in H$ is such that $\|h\|_H=1$. For $m\neq n$, $\cH_n$ and $\cH_m$ are orthogonal and 
$L^2(\eta)$ coincides with 
the direct orthogonal sum of the $\cH_n$'s, i.e. $L^2(\eta)=\bigoplus_{n}\cH_n$ (see~\cite[Theorem 1.1.1]{Nualart2006}). 
The subspace $\cH_n$ is called the {\it $n$-th homogeneous Wiener chaos}. 

When the Hilbert space $H$ is of the form $L^2(T)$, for $(T,\cB,\mu)$ a measure space with a 
$\sigma$-finite and atomless measure $\mu$, the decomposition above can be refined. Namely, 
for every $n\in\N$ there exists a canonical contraction $\sint{}:\bigoplus_{n\ge 0} L^2(T^n) \to L^2(\eta)$, 
called (iterated) Wiener-It\^o integral with respect to $\eta$, 
which restricts to an isomorphism $\sint{}:\fock \to L^2(\eta)$ on the Fock space $\fock:=\bigoplus_{n \ge 0} \fock_n$, 
where $\fock_n$ denotes the space $L_\sym^2(T^n)$ of functions in $L^2(T^n)$ which are symmetric
with respect to permutation of variables. Moreover, the restriction of $\sint$ to $\fock_n$, denoted by $\sint_n$, is an isomorphism 
onto the $n$-th homogenous Wiener chaos $\cH_n$ since, by~\cite[Proposition~1.1.4]{Nualart2006}, we have 
\begin{equation}\label{eq:Hermite}
n! H_n(\eta(h))= \sint_n(\hot)\,,\qquad\text{for all $h\in H$ such that $\|h\|_H=1$,}
\end{equation}
where $\hot$ is the tensor product of $n$ copies of $h$. 
We also recall \cite[Proposition~1.1.3]{Nualart2006} that for $f\in L_\sym^2(T^n)$ and $g\in L_\sym^2(T^m)$ one has
\begin{equation}\label{eq:multiplicationiterated}
\sint_n(f) \sint_m(g)  = \sum_{p=0}^{m\land n} p! \binom{n}{p}\binom{m}{p} 
\sint_{m+n-2p} (f \otimes_p g),
\end{equation} 
where
\[
(f \otimes_p g)(x_{1: m+n-2p})\eqdef 
\int_{T^p}\mu(dy_1) \dots \mu(dy_p) f(x_{1:n-p}, y_{1:p})
g(x_{n-p+1:m+n-2p}, y_{1:p}). 
\]
Here, we adopted the short-hand notation $(x_{1:n})\eqdef (x_1,\dots,x_n)$. 
\newline

We call a function $F:\CD'\to\R$ a {\it cylinder function} if there exist $\varphi_1,\dots,\varphi_n\in \CD$ 
and a smooth function $f:\R^n\to\R$
with all partial derivatives growing at most polynomially at infinity such that
\def\eBprime{u}$F(\eBprime)=f(\eBprime(\varphi_1),\dots,\eBprime(\varphi_n))$.
Given a cylinder function $F$ as above, we define its ``directional derivative''
in the direction of $\psi$ by $D_{\psi}F(\eBprime)\eqdef \sum_{i=1}^n \partial_i f(\eBprime(\varphi_1),\dots,\eBprime(\varphi_n)) \langle \vphi_i, \psi \rangle$. 
In case that $\{\vphi_i\}_{i \le n}$ forms an orthonormal system in $L^2(\R^n)$ one has the simplified formula $D_{\vphi_i} F(u) = \partial_i f(\eBprime(\varphi_1),\dots,\eBprime(\varphi_n))$.

Similarly, given Hilbert space $H$ and an isonormal Gaussian process $\eta$ on $H$, we call a random variable $X \in L^2(\eta)$ ``smooth'' (compare \cite[(1.28)]{Nualart2006}), if there exists $h_1 , \dots , h_n \in H$ and a smooth function $f :\R^n \to \R$ with all derivatives growing at most polynomially, such that $X=f(\eta(h_1),\dots,\eta(h_n))$ almost surely. 
For a smooth random variable $F$, we define the Malliavin derivative (see~\cite[Definition 1.2.1]{Nualart2006}) of $X$ by
\begin{equ}
DX\eqdef \sum_{i=1}^nD_{h_i} X(\eta) h_i=\sum_{i=1}^n\partial_if(\eta(h_1),\dots,\eta(h_n)) h_i\,.
\end{equ}
In order to manipulate Malliavin derivatives, an important property is the analog of the integration by parts formula, 
the so called {\it Gaussian integration by parts} given in~\cite[Lemma 1.2.2]{Nualart2006}. 
Let $F,G$ be smooth random variables on $\Omega$, then 
\begin{equation}\label{eq:intbyparts1}
\E[G\langle DF, h\rangle]= \E[-F\langle DG, h\rangle + FG\eta(h)],
\end{equation}
where $\E$ is the expectation with respect to the law of $\eta$. 
\newline

Throughout the rest of the paper, the isonormal Gaussian process $\eta$ we will consider 
is the zero-mean spatial white noise on the two dimensional 
torus $\T^2$. To be more precise, $\eta$ is a centred isonormal Gaussian process on $H\eqdef L^2_0(\T^2)$ 
whose covariance function is given by 
\begin{equ}\label{eq:spatial:white:noise}
\E[\swn(\vphi)\swn(\psi)]=\langle \vphi, \psi\rangle
\end{equ}
for any two functions $\varphi,\psi\in H$, where $\langle\cdot,\cdot\rangle$ is the usual scalar product in $L^2(\T^2)$. 
We will mainly work with the Fourier representation of $\eta$, given by the 
family of complex valued, centred Gaussian random variables $\{\eta_k\}_{k\in\Z^2}$, where $\eta_0=0$, $\overline{\eta_k}=\eta_{-k}$
and  $\E[\eta_k\eta_j]=\delta_{k+j=0}$. Given a smooth random variable 
$F$ in $L^2(\eta)$ (or a cylinder function on $\CS'(\T^2)$) and $k\in\Z^2$ we write 
\begin{equation}\label{e:FDer}
D_k F\eqdef \langle DF,e_k\rangle_{L^2(\T^2;\C)}
\end{equation}
where for any $f,g\in L^2(\T^2;\C)$, the scalar product above is the usual Hermitian scalar product, i.e. 
$\langle f,g\rangle_{L^2(\T^2;\C)}=\int f\bar{g}$. For future use, we remark here that for any $k\in\Z^2$ 
\begin{equation}\label{eq:intbyparts}
\E[G D_kF]=\E[G\langle DF, e_k\rangle_{L^2(\T^2;\C)}]= \E[-F D_kG + FG\eta_{k}].
\end{equation}
where we used~\eqref{eq:intbyparts1} together with~\eqref{e:complexPairing} and~\eqref{e:FT}.


\section{Properties of the approximating equations}
\label{S:Properties}

In order to simplify our analysis below we will be working with $u^N\eqdef (-\Delta)^{\frac{1}{2}}h^N$,
which solves
\begin{equation}\label{e:AKPZ:u}
\partial_t u^N = \frac{\donst}{2} \Delta u^N
+
\const \cN^N[u^N] + \donst^{\frac{1}{2}}(-\Delta)^{\frac{1}{2}}\xi
\end{equation}
where the non-linearity $\cN^N$ is given by 
\begin{equation}\label{e:nonlin}
\cN^N[u^N]\eqdef (-\Delta)^{\frac{1}{2}}\Pi_N \Big((\Pi_N \partial_1(-\Delta)^{-\frac{1}{2}} u^N)^2 - (\Pi_N \partial_2 (-\Delta)^{-\frac{1}{2}} u^N)^2\Big)\,.
\end{equation}
By definition of the H\"older-Besov spaces (\ref{def:Besov}), the fractional Laplacian \eqref{e:fLapla} is a continuous 
and continuously invertible linear bijection $(-\Delta)^\half:\CC^{\alpha}_0 \to \CC^{\alpha-1}_0$, for any $\alpha\in\R$, 
where $\CC^{\alpha}_0$ denotes the closed subspace of $\CC^{\alpha}$ spanned by distributions with vanishing 
$0$-th Fourier component. 
Theorem \ref{thm:tight} therefore reduces to showing tightness of the sequence $u^N$ in $C_T^\gamma\CC^{\alpha-1}$
 (see Theorem~\ref{thm:tightness}), for 
$\alpha$ as in the statement, and of the $0$-th Fourier mode (see Theorem~\ref{thm:tightnessAKPZ}. 
As part of this argument we also show that the anisotropic KPZ equation~\eqref{e:kpz:reg} requires 
no renormalisation other than the coupling constant renormalisation introduced in Theorem \ref{thm:tight} 
(see Remark~\ref{rmk:NoRenom}).   

Passing to Fourier variables, we see that equation~\eqref{e:AKPZ:u} can be equivalently written as an infinite system 
of SDEs 
\begin{equation}\label{e:kpz:u}
\dd u^N_k =
\left(-\frac{\donst}{2} |k|^2 u_k^N
+
\const \cN^N_k[u^N]\right) \dd t + \donst\phalf |k|\dd B_k(t)\,,\qquad k\in\Z^2_0,
\end{equation}
where the complexed valued Brownian motions $B_k$ are defined via  $B_k(t)\eqdef \int_0^t \xi_k(s)\, ds$, which in particular implies that $\overline{B_k}= B_{-k}$. Hence their quadratic covariation is given by
$\dd\langle B_k, B_\ell \rangle_t = \mathds{1}_{\{k+\ell=0\}}\, \dd t$ for $k,l \ne 0$. The $k$-th Fourier component of the non-linearity is
\begin{align}
\cN_k^N[u^N]&\eqdef\cN^N[u^N](e_{-k})= \sum_{\ell+m=k}\nonlin_{\ell,m} u_\ell^N u_m^N\,,\label{e:nonlinF}\\
 \nonlin_{\ell,m} &\eqdef |\ell+m| \frac{\thec(\ell,m)}{|\ell||m|} \indN{\ell,m,\ell+m} \label{e:nonlinCoefficient}
\end{align}
where, for $\ell=(\ell_1,\ell_2),\,m=(m_1,m_2)\in\Z^2_0$,  $c(\ell,m) \eqdef \ell_2m_2- \ell_1m_1$ and 
$\indN{a,b,\ldots}$ is a abbreviation for 
$\mathds{1}_{|a|_\infty\le N, |b|_\infty\le N, \ldots}$.
 
This approximation scheme has the advantage that it completely decouples the equations for 
$\{u_k^N\,:\,|k|_\infty\leq N\}$ and $\{u_k^N\,:\,|k|_\infty> N\}$. The latter is an infinite family of independent 
Ornstein-Uhlenbeck processes, while the first is a finite-dimensional system of SDEs interacting via a quadratic nonlinearity. 
Local existence and uniqueness is classical (since the coefficients are locally Lipschitz continuous) and the process
$t\mapsto \{u^N_k(t)\}_{k\in\Z_0^2}$ is clearly strong Markov.
At this point we refrain from being more specific about the state space for this process. 
As long as we are working with fixed $N$, any ``reasonable'' choice could be used for the sake of the current section 
(one could take e.g. $H^\alpha$, $\alpha<-1$, if one wants to deal with~\eqref{e:AKPZ:u} directly 
or $\R^{\Z_0^2}$ with the product topology, if instead one focuses on the system~\eqref{e:kpz:u}). 
We postpone a detailed discussion of the spaces we actually want to work in to the proof of tightness 
in Section~\ref{sec:Tightness}.
\newline

Returning to the system in~\eqref{e:kpz:u}, we can easily determine the generator $\gen$ for the dynamics of $u^N$
(for example by applying It\^o's formula, single out the drift part, see~\eqref{e:Ito} below, and take the Fourier transform).
Let $F$ be a real valued cylinder function acting on distributions $v \in \CD'(\T^2)$ and depending only on finitely many Fourier components of the distribution $v$,
and decompose the generator into $\gen=\gens+\gena$, where $\gens$ and $\gena$ act on $F$ via
\begin{align}
(\gens F)(v) &\eqdef \frac{\donst}{2} \sum_{k \in \Z^2} |k|^2 (-v_{-k} D_k  +  D_{-k}D_k )F(v) \label{e:gens}\\
(\gena F)(v) &= \const \sum_{m,l \in \Z_0^2} \nonlin_{m,l} v_m v_l D_{-m-l} F(v)\,. \label{e:gena}
\end{align}
As a first step of our analysis, we show that the spatial white noise $\eta$ on $\T^2$
 is invariant for the Markov process $u^N$ for all $N\in\N$.

\begin{lemma}\label{lem:invariant}
For any $N\in\N$, the spatial white noise $\eta$ defined in (\ref{eq:spatial:white:noise}) is invariant for the solution 
$u^N=\{u^N_k\}_{k\in\Z_0^2}$ of~\eqref{e:kpz:u}. 
Moreover, with respect to $L^2(\eta)$ the symmetric and anti-symmetric part of $\gen$ are given by 
$\gens$ and $\gena$, respectively.
\end{lemma}

\begin{proof}
According to~\cite{Ech}, for the first statement it is enough to prove that $\E [\gen G(\swn)] =0$ 
for all $C^2$ cylinder functions and 
we will prove the above relation for $\gens$ and $\gena$ separately, beginning with the first. 
Let $G$ be a cylinder function depending on only finitely many Fourier components of $\eta$. 
In the Gaussian integration by parts formula~\eqref{eq:intbyparts} set $F= D_{k}G$ and 
$G\equiv 1$, so that we have
\begin{align*}
\E[D_{-k} D_{k} G(\eta)] 
&= \E[D_{k} G(\eta) \eta_{-k}],
\end{align*}
from which $\E\gens G(\swn)=0$ follows.
For the operator $\gena$ we use again Gaussian integration by parts (this time with $F=G$ and $G(\eta)=g(\eta_m,\eta_\ell)=\eta_m\eta_\ell$ in the notation of~\eqref{eq:intbyparts}) to obtain 
\begin{equation}
\label{e:antisym}\begin{split}
\E[ \eta_m \eta_\ell D_{-m-\ell}G(\eta) ]
&= \E[\eta_m \eta_\ell \eta_{-m-\ell} G(\eta) - G(\eta) D_{-m-\ell} (\eta_m \eta_\ell)]\\
&=\E[\eta_m \eta_\ell \eta_{-m-\ell} G(\eta)]
\end{split}
\end{equation}
where the last passage is a consequence of the choice $\eta(e_0)=0$. 
Now, the function $G$ on the right hand side does not depend on either $m$ or $\ell$. We claim that the following stronger
statement holds, for any $\{\eta_k\}_{k\in\Z^2_0}$ we have 
\begin{equation}\label{e:simmetry}
\sum_{m,\ell \in \Z_0^2} \nonlin_{m,\ell} \eta_m \eta_\ell \eta_{-m-\ell}=0\,.
\end{equation}
Assuming~\eqref{e:simmetry}, upon multiplying by $\nonlin_{m,l}$ and summing over all $\ell$ and $m\in\Z^2_0$ 
at both sides of~\eqref{e:antisym}, we would obtain $\E\gena G(\swn)=0$ and thus conclude the proof of invariance of white noise
for the dynamics of $u^N$.
 
Let us prove~\eqref{e:simmetry} (an alternative proof is provided in Appendix~\ref{sec:alter}). 
Observe that $f(m,\ell)\eqdef\eta_m \eta_\ell \eta_{-m-\ell}$ satisfies the symmetry relation
$f(m,\ell) = f(-m-\ell,\ell) = f(m,-m-\ell)$, so that it suffices to check that $\nonlin_{m,\ell}$ is antisymmetric once we sum over all 
permutation of $m,\ell$ and $-m-\ell$. We compute
\begin{align}\label{e:gena:antisym:prf:1}
\nonlin_{m,l} =  -\frac{(m_1+\ell_1)^2m_1 \ell_1}{|m+\ell ||m||\ell |}
+
\frac{(m_2+\ell_2)^2 m_2 \ell_2}{|m+\ell ||m||\ell |}
+
\frac{(m_1+\ell_1)^2 m_2 \ell_2}{|m+\ell ||m||\ell |}
-
\frac{(m_2+\ell_2)^2 m_1 \ell_1}{|m+\ell ||m||\ell |}
\end{align}
where $\ell_1,\ell_2$ and $m_1,m_2$ are the components of $\ell$ and $m$ respectively. 
Denote these summands by $\nonlin_{m,\ell,(i)}$, $i=1,2,3,4$. 
Then, for $i=1,2$, $\nonlin_{m,\ell,(i)}+\nonlin_{-m-\ell,\ell,(i)}+\nonlin_{m,-m-\ell,(i)}=0$, while
\[
\nonlin_{m,\ell,(3)} + \nonlin_{-m-\ell,\ell,(3)} + \nonlin_{m,-m-\ell,(3)}
=
2m_1\ell_1 m_2 \ell_2 - m_1^2 \ell_2^2 - \ell_1^2 m_2^2,
\]
which cancels the corresponding term coming from $\nonlin_{m,l,(4)}$. 
\newline

We next show that \eqref{e:gens} and \eqref{e:gena} are indeed the symmetric and antisymmetric part of $\gen$. 
The first claim follows directly from 
\begin{align*}
\E[ D_{-k} D_{k} F(\eta) G(\eta)]
&=
\E[D_{k} F(\eta) G(\eta) \eta_{k}]
-
\E[D_{k} F(\eta) D_{-k} G(\eta)]
\end{align*}
so that
\[
\E[\gens F(\eta) G(\eta) ]
=
-\frac{\donst}{2} \sum_{k \in \Z^2} |k|^2 \E[D_{k}F(\eta) D_{-k} G(\eta)],
\]
and the latter is clearly symmetric. For the antisymmetric part we compute
\begin{align*}
\E[ \eta_m \eta_\ell D_{-m-\ell} F(\eta) G(\eta)]
=
\E[\eta_m \eta_\ell \eta_{-m-\ell} F(\eta) G(\eta) ]
-\E[ F(\eta) D_{-m-\ell}(G(\eta)\eta_m \eta_\ell) ]\,.
\end{align*}
Notice that, summing up over $\ell,m\in\Z_0^2$ the first term on the right hand side drops out for the same 
reason as in the proof of stationary. We can apply the Leibniz rule to the second and, recalling that we chose $\eta(e_0)=0$, we get 
\[
\E[\gena F(\eta) G(\eta)]
=
-\const \sum_{m,\ell \in \Z^2} \nonlin_{m,\ell} \E[\eta_m \eta_\ell F(\eta) D_{-m-\ell} G(\eta)]
=
-\E[ F(\eta) \gena G(\eta)]\,,
\]
so that the proof is concluded. 
\end{proof}

\begin{remark}
The previous lemma provides the second advantage of our approximation scheme, namely the fact that the invariant measure
of $u^N$ is independent of $N$. If we decided, in addition to smoothing the nonlinearity, to cut the high Fourier modes 
of the space-time white noise $\xi$ appearing in~\eqref{e:AKPZ:u} (i.e. replace $\xi$ by $\Pi^N\xi$), 
so that $u^N_k\equiv 0$ for every $k$ with $|k|_\infty\geq N$, then the same proof shows that the invariant 
measure would be $\Pi^N\eta$.
\end{remark}

Thanks to the previous lemma and classical theory of stochastic analysis, we immediately deduce the following proposition. 

\begin{proposition}
For any deterministic initial condition $u^N(0)=\{u_k^N(0)\}_{k\in\Z^2_0}$, the solution $t\mapsto u^N(t)=\{u_k^N(t)\}_{k\in\Z^2_0}$
of~\eqref{e:kpz:u} exists globally in time. 
\end{proposition}
\begin{proof}
The proof is a consequence of the fact that the system~\eqref{e:kpz:u} has an invariant measure with finite moments of 
all orders and satisfies the strong Feller property. 
The latter is obvious for $\{u^N_k\,:\,|k|_\infty> N\}$, since they are independent Ornstein-Uhlenbek
processes, while for $\{u^N_k\,:\,|k|_\infty\leq N\}$ it follows by~\cite{DPEZ}.
\end{proof}

\begin{remark}
The equality~\eqref{e:simmetry} in the proof of Lemma~\ref{lem:invariant} shows that the nonlinearity~\eqref{e:nonlin} 
is {\it dissipative}. This fact can be used to 
prove the previous proposition directly at the level of the stochastic PDE~\eqref{e:AKPZ:u} either following the strategy
in~\cite[Section 4]{GubinelliJara2012}, or using the fact that, by Lemma~\ref{lem:invariant}, the equation has an invariant measure,  together with the strong Feller property. 
\end{remark}

Similarly to~\cite{GPGen} we want to improve our understanding of the generator associated to $u^N$. More specifically, 
we would like to know how $\gen$ acts on elements of $L^2(\eta)$ and ensure that $\gen$ is reasonably well behaved 
when applied to elements belonging to a homogeneous Wiener chaos.

For that, recall that the Fourier transform $\CF$ maps $\fock_n= L^2_\sym(\T^{2n})$ (isometrically) 
into $ \ell^2((\Z^2)^n)$, i.e. $\CF(\cdot)=\hat\cdot : L^2_\sym(\T^{2n}) \to \ell^2((\Z^2)^n)$. 
Moreover, if $\CO$ is an operator acting on (a subspace of) $L^2(\swn)$, we will denote by 
$\mathfrak  O$ the operator on $\fock$ such that for all 
$\varphi\in \fock$ one has $\CO \sint(\varphi)= \sint(\mathfrak O\varphi)$.

\begin{lemma}\label{lem:generator:fockspace}
For any $n \in \N$ the operator $\gens$ leaves $\wc^n$ invariant and $\gena$ maps $\wc^n$ into $\wc^{n-1}\oplus \wc^{n+1}$, and one has for any $K \in \fock_n$ the identity
\begin{align}\label{e:gens:fock}
\gens \sint_n	K = \frac{\donst}{2} \sint_n \Delta K.
\end{align}
Moreover, one can write $\gena = \genap+ \genam$, where $\genap$ increases and $\genam$ decreases the order of the Wiener chaos by one, 
i.e. $\genap:\wc^n\to\wc^{n+1}$ and $\genam:\wc^n\to\wc^{n-1}$, and
their action in Fock space representation, denoted by $\genapf$ and $\genamf$ respectively, satisfy
\begin{align}\label{e:genap:fock}
\CF(\genapf K)(k_{1:n+1}) 
&= n \const 
\nonlin_{k_1, k_2} \hat K(k_1 + k_2, k_{3:n+1})  
\\
\label{e:genam:fock}
\CF(\genamf K)(k_{1:n-1})
&=
2n(n-1) \const
\sum_{\ell+m = k_1}
\nonlin_{k_1, -\ell} \hat K(\ell,m, k_{2:n-1}),
\end{align}
where we used the short-hand notation $k_{1:n+1}= (k_1,\ldots, k_{n+1})$. 
Finally, the operator $-\genap$ is the adjoint of $\genam$ in $L^2(\swn)$.
\end{lemma}
\begin{proof}
It suffices to show \eqref{e:gens:fock}, \eqref{e:genap:fock} and \eqref{e:genam:fock} for a kernel of the type $K = \hot$ 
for some $h \in H$. 
As a consequence of~\eqref{eq:Hermite} and the fact that Hermite polynomials satisfy $H_n' = H_{n-1}$
(see~\cite[Equation (1.2)]{Nualart2006}), the Malliavin derivatives of stochastic integrals of such kernels can be written as
\begin{align*}
D_k \sint_n(\hot)
&=
n \sint_{n-1}(\hotm) h_{k}
\\
D_{-k} D_k \sint_n(\hot)
&=
n(n-1) \sint_{n-2}(\hotmm) h_k h_{-k}\,.
\end{align*}
We start by analysing $\gens$. To show \eqref{e:gens:fock}, first note that
\begin{equation}
\sint_1(\Delta h) = I_1\Big(-\sum_{k\in\Z^2}|k|^2h_ke_k\Big)= 
-\sum_{k\in\Z^2}|k|^2 h_k \sint_1(e_k)= -\sum_{k\in\Z^2}|k|^2 h_k \swn_{-k},
\end{equation}
where the last equality follows from~\eqref{eq:Hermite}, and the fact that $H_1$ is the identity.
Using the above observation and~\eqref{eq:multiplicationiterated}, we thus see that
\begin{align}
-\sum_{k\in \Z^2} |k|^2 \swn_{-k} &D_k I_n(\hot) 
=
n \sint_{n-1}( \hotm )
\sint_1 (\Delta h) \nonumber
\\
&=
n\sint_n(\hotm \otimes \Delta h)
+
n(n-1)
\sint_{n-2}(\hotmm) \langle h, \Delta h\rangle. \label{e:LzeroPart1}
\end{align}
On the other hand, a similar calculation shows 
\[
\sum_{k \in \Z^2} |k|^2 D_k D_{-k} \sint_n(\hot)
=
-n(n-1) \sint_{n-2}(\hotmm)\langle h, \Delta h\rangle\,.
\]
Hence, the first summand in~\eqref{e:LzeroPart1} drops out and we obtain the identity~\eqref{e:gens:fock}. 
For the operator $\gena$, proceeding as above we see that 
\begin{align*}
\gena  I_n (\hot)
&=
\const 
n\sum_{\ell,m\in\Z^2_0} \nonlin_{\ell,m} \eta_\ell\eta_m \sint_{n-1}(\hotm)h_{-m-\ell}\\
&=
\const n \sint_2\left(\sum_{\ell,m\in\Z^2_0}\nonlin_{\ell,m}\,h_{-m-\ell}\,e_{-m} \otimes e_{-\ell}\right) \sint_{n-1}(\hotm)\,.
\end{align*} 
By the product rule~\eqref{eq:multiplicationiterated}, we get
\begin{align*}
\gena  I_n (\hot)
=&
\const
n \sint_{n+1}\left(\sum_{\ell,m} \nonlin_{l,m}  h_{-\ell-m}\, \left(e_{-l} \otimes e_{-m} \otimes \hotm\right)\right)
\\
&
+\const
2n(n-1) \sint_{n-1}\left(\sum_{\ell,m} \nonlin_{\ell,m} h_{\ell} h_{-\ell-m}\, \left(e_{-m} \otimes \hotmm\right) \right)
\\
&
+\const
n(n-1)(n-2) \left(\sum_{\ell,m} \nonlin_{\ell,m} h_{\ell} h_{m} h_{-\ell-m}\right) \sint_{n-3}(\otimes^{n-3} h). 
\end{align*}
Notice at first that the third summand disappears thanks to~\eqref{e:simmetry}. Moreover, the first and 
second summand live in $\cH_{n+1}$ and $\cH_{n-1}$ respectively and, upon taking the Fourier transform of the integrands, 
we immediately obtain \eqref{e:genap:fock} and \eqref{e:genam:fock}. 

At last, $\genam$ and $-\genap$ are adjoint to each other since $\gena=\genam + \genap$ is an anti-symmetric operator 
on $L^2(\mu)$ and, as noted above, one has $\genam: \wc^n \to \wc^{n-1}$ and $\genap : \wc^n \to \wc^{n+1}$. Indeed, 
if $F(\eta)=\sum_n \sint_n(f_n)$ and $G(\eta)=\sum_n \sint_n(g_n)$, then
\begin{equs}
\E[\genap F(\eta) G(\eta)]&=\sum_n \E[\genap \sint_n(f_n)\sint_{n+1}(g_{n+1})]=\sum_n \E[\gena \sint_n(f_n)\sint_{n+1}(g_{n+1})]\\
&=-\sum_n \E[ \sint_n(f_n)\gena\sint_{n+1}(g_{n+1})]=-\sum_n \E[ \sint_n(f_n)\genam\sint_{n+1}(g_{n+1})]\\
&=-\E[ F(\eta) \genam G(\eta)]
\end{equs}
and the proof is concluded. 
\end{proof}

\section{Upper bounds and tightness of the approximating sequence}\label{sec:Tightness}

In this section we want to show how to obtain suitable bounds 
(depending on the coupling constants $\const$ and $\donst$) on the time integral of 
(non-linear) functionals of the solution $u^N$ of~\eqref{e:AKPZ:u}. The point we want to make is that the technique exploited 
in~\cite{GubinelliJara2012} is sufficiently flexible to be able to handle even cases in which the limiting equation 
is {\it critical}. 

To get a feeling of the procedure followed in the aforementioned paper, consider a generic functional $F$ 
in the domain of the generator $\cL^N$ of the Markov process $\{u^N(t)\}_{t\in\R_+}$ solving~\eqref{e:kpz:u}, whose 
symmetric and antisymmetric part, with respect to the invariant measure $\swn$, are $\gens$ and $\gena$ respectively 
(see Lemma~\ref{lem:invariant}). 
The main idea is that the relation between the forward 
and the backward processes ($u^N(t)$ and $u^N(T-t)$)
can be used in the representation of $F(u^N)$ given by Dynkin's (or It\^o's) formula (see~\eqref{e:Ito} and~\eqref{e:Ito:backward})
in order to get rid of both 
the boundary terms and the terms containing $\gena F(u^N)$. In this way,
the time average of $\gens F(u^N)$ can be expressed as the sum of two martingales (see~\eqref{eq:Ito}) 
which in turn can be controlled 
via their quadratic variation. 
The latter is explicit and depends only on $u^N$ evaluated at a {\it single point in time}. 
The knowledge of the invariant measure for the process is then the key to obtain a bound on (moments of)
the quadratic variation of these martingales (see Lemma~\ref{lem:ito-trick}). At last, once estimates for quantities of the form 
$\int_0^T\gens F(u^N(s))\dd s$ are available, analogous estimates for $\int_0^T V(u^N(s))\dd s$, for more general functionals $V$, 
can be consequently achieved if one is able to determine a solution $F$ to the {\it Poisson equation} given by 
\begin{equation}\label{e:Poisson}
\gens F = V\,.
\end{equation}
In what follows, we will first describe in more detail the strategy outlined above and then
show how we can take advantage of these techniques in the context of the anisotropic KPZ equation. 
\newline

Let $F=F(t,\cdot)$ be a cylinder function depending smoothly 
on time. Thanks to It\^o's formula (and the Fourier representation of $u^N$ given in~\eqref{e:kpz:u}), we can write
\begin{equation}\label{e:Ito}
F(t,u^N(t)) 
= F(0,u^N(0))+
\int_0^t\, \left(\partial_s +\gen\right)F \,(s,u^N(s))\, \dd s + \donsth \martf_t(F),
\end{equation}
where $\martf _\cdot(F)$ is the martingale (depending on $F$) defined by
\begin{equation}\label{e:Mart}
\dd \martf_t(F)
=
\sum_{k\in\Z^2}
(D_k F)(t, u^N(t) )  |k| \dd B_k(t)
\end{equation}
whose quadratic variation is
\begin{align}\label{e:energy}
\dd\qvar{\martf (F)}_t
=
\energy(F)(t,u^N(t))\, \dd t
\eqdef
\sum_{k\in\Z^2}  |k|^2
|D_k F(t,u^N(t))|^2\,
\dd t\,.
\end{align}
For fixed $T>0$, it follows from Lemma~\ref{lem:invariant} that the backward process $\bar u^N(t) \eqdef u^N(T-t)$ 
is itself a Markov process whose generator is given by the adjoint of $\gen$, $(\gen)^\ast=\gens - \gena$. In particular, 
applying again It\^o's formula, but this time on $F(t,\bar u^N(t))$, we get
\begin{align}\label{e:Ito:backward}
F(t,\bar u^N(t)) = F(0,\bar u^N(0)) + \int_0^t (\partial_s+ \gens - \gena)F (s,\bar u^N(s))\, ds + \donsth \martb_t(F),
\end{align}
where $\martb(F)$ is a martingale with respect to the {\it backward} filtration, generated by the process $\bar u^N$, 
and its  quadratic variation is given by $\dd\qvar{\martb(F)}_t
=
\energy(F)(T-t,u^N(T-t))\, \dd t
$.
Summing up (\ref{e:Ito}) and (\ref{e:Ito:backward}) one obtains the following analog of~\cite[eq. (10)]{GubinelliJara2012}
\begin{equation}\label{eq:Ito}
2 \int_0^t \gens F (s,u^N(s)) ds
=
\donsth(-\martf_t(F) + \martb_{T-t}(F) - \martb_{T}(F))\,.
\end{equation}
The right hand side of~\eqref{eq:Ito} can be bounded by the Burkholder-Davis-Gundy inequality and
this yields~\cite[Lemma 2]{GubinelliJara2012} (which in turn was inspired by~\cite[Lemma 4.4]{CLO}), that we here recall. 

\begin{lemma}\label{lem:ito-trick}
\emph{(It\^o-trick)}
For any $p\geq 2$, $T>0$ and cylinder function $F=F(t,\cdot)$ smoothly depending on time, the following estimate holds
\begin{align}\label{e:ito-trick}
\Exp \Big[ \sup_{t \le T} \Big|\int _0 ^t \gens F(s,u^N(s)) \dd s\Big|^p\Big]^\frac{1}{p}
\lesssim
\donst^{\frac{1}{2}}
T^{\frac{1}{2}} \sup_{s\in[0,T]}\E \left[| \energy (F)(s,\swn) |^{\frac{p}{2}}\right]^{\frac{1}{p}}\,.
\end{align}
Moreover, in the specific case in which $F(t,x)=\sum_{i\in I}e^{a_i(T-t)}\tilde F_i(x)$, where $I$ is an index set, $a_i\in \R$ 
and $F_i$ is a cylinder function for every $i\in I$, we have 
\begin{equation}\label{e:itoconvolution}
\begin{split}
\Exp \Big[  \Big|\int _0 ^T \sum_{i\in I}e^{a_i(T-s)}&\gens \tilde F_i(u^N(s)) \dd s\Big|^p\Big]^\frac{1}{p}\\
&\lesssim
\donst^{\frac{1}{2}}
\left(\sum_{i\in I}\Big(\frac{e^{2a_iT}-1}{2a_i}\Big) \E\left[ | \energy (\tilde F_i)(\swn) |^{\frac{p}{2}}\right]^{\frac{2}{p}}\right)^{\frac{1}{2}}\,.
\end{split}
\end{equation}
In both cases, the proportionality constant hidden in $\lesssim$ is independent of both $N$ and $F$.   
Here and below we use the symbol $\Exp$ to denote expectations with respect to the law of $\{u^N(t)\}_{t\in\R_+}$ and $\E$ 
for expectations with respect to the law of $\eta$.
\end{lemma}

\begin{remark}
The crucial aspect of the previous lemma is that we are able to bound the expectation of functionals of $u^N$ with respect 
to the {\it space-time} law of $u^N$ in terms of the expectation with respect to the {\it sole} invariant measure, so that 
explicit computations become indeed possible. 
\end{remark}

\begin{proof}
The proof of~\eqref{e:ito-trick} is that of equation (11) in~\cite[Lemma 2]{GubinelliJara2012}. 
The second bound can be obtained following the proof of~\cite[Lemma 2, eq (12)]{GubinelliJara2012} and 
we provide the details for completeness.
Notice that, given $t\in[0,T]$ and $F$ as in the statement, the left hand side of~\eqref{e:itoconvolution} is bounded from above by 
\begin{align*}
\Exp \Big[  \sup_{t\leq T}\Big|\int _0 ^t \sum_{i\in I}e^{a_i(T-s)}\gens &\tilde F_i(u^N(s)) \dd s\Big|^p\Big]^\frac{1}{p}=
\Exp \Big[  \sup_{t\leq T}\Big|\int _0 ^t \gens F(s,u^N(s)) \dd s\Big|^p\Big]^\frac{1}{p}\\
&\lesssim \donst^{\frac{1}{2}}\Exp \Big[\Big|\int_0^T\energy(F(s,\cdot))(u^N(s))\dd s\Big|^{\frac{p}{2}}\Big]^{\frac{1}{p}}\\
&=
\donst^{\frac{1}{2}}\Exp\Big[\Big|\int_0^T  \sum_{i\in I}e^{2a_i(T-s)}\energy( \tilde F_i)(u^N(s))\dd s\Big|^{\frac{p}{2}}\Big]^\frac{1}{p}
\end{align*}
where for the first equality, we used the fact that $\gens$ acts only on the spatial variable, the subsequent bound follows 
by Burkholder-Davis-Gundy inequality and the last comes from the fact that 
$\dd \langle M_\cdot(F)\rangle_t=\sum_{i\in I}e^{2a_i(T-t)}\energy( \tilde F_i)(u^N(t))\dd t$. The right hand side of the latter 
is trivially bounded by 
\begin{multline*}
\left(\int_0^T\sum_{i\in I}e^{2a_i(T-s)} \Exp\left[\left|\energy( \tilde F_i)(u^N(s))\right|^{\frac{p}{2}}\right]^\frac{2}{p}\dd s\right)^\half
\\
=\donst^{\frac{1}{2}}
\left(\sum_{i\in I}\int_0^Te^{2a_i(T-s)}\dd s \E\left[ | \energy (\tilde F_i)(\swn) |^{\frac{p}{2}}\right]^{\frac{2}{p}}\right)^{\frac{1}{2}}
\end{multline*}
the equality, coming from the fact that $u^N$ appears only evaluated at a single point in time, and its law is that of $\eta$. 
By evaluating the integral we obtain~\eqref{e:itoconvolution}.
\end{proof}

%
%
%

For any test function $\phi$ we are interested in uniform bounds of the linear and the non-linear part of~\eqref{e:AKPZ:u} tested against $\phi$, i.e. on $\const\cN^N[\eta](\varphi)$ and $\frac{\donst}{2}\eta(\Delta\varphi)$, respectively. As mentioned 
above, they will be obtained by combining the It\^o-trick, Lemma \ref{lem:ito-trick}, with an explicit solution of the Poisson equation. 

Thanks to the Fourier representation of the non-linearity given in~\eqref{e:nonlinF}, 
we can write $\cN^N[\eta](\varphi)$ as a second order Wiener-It\^o integral of the form
\begin{equ}
\cN^N[\eta](\varphi)=\sum_{\ell,m\in\Z^2_0}  \nonlin_{\ell,m}
\eta_l \eta_m \varphi_{-\ell-m}=
\sint_2\left(\sum_{\ell,m\in\Z^2_0} 
\nonlin_{\ell,m}
\varphi_{-\ell-m} e_{-\ell}\otimes e_{-m}\right)\,.
\end{equ}
Using~\eqref{e:gens:fock} it is easy to see that the solution of the Poisson
equation $\gens \pois[\eta](\varphi)= \const\cN^N[\eta](\varphi)$ is the cylinder function (clearly, depending on $N$) given by 
\begin{equation}\label{e:solPoiNL}
\pois[\eta](\varphi) \eqdef  2\const\donst^{-1}\sum_{\ell,m\in\Z^2_0}  
\frac{\nonlin_{\ell,m}}{|\ell|^2 + |m|^2} 
\eta_l \eta_m \varphi_{-\ell-m}\,.
\end{equation}
On the other hand, the solution $\poisl[\eta](\varphi)$ of $\gens\poisl[\eta](\varphi)=\frac{\donst}{2}\eta(\Delta\varphi)$ is, 
again by~\eqref{e:gens:fock}, simply $\poisl[\eta](\varphi)=\eta(\varphi)$. 
We are now ready to state and prove the following lemma. 

\begin{lemma}\emph{(Energy estimates)}\label{lem:energy-estimate}
Let $T>0$ be fixed, $\varphi\in H^1$ and $u^N$ be the solution to~\eqref{e:AKPZ:u}. Let $\cN^N$ be defined according 
to~\eqref{e:nonlin}. Then, for any $p\geq 2$ the following estimates hold
\begin{align}
\Exp \left[ \sup_{t \le T} \Big|\int _0 ^t  \const \cN^N[u^N(s)](\varphi) \dd s\Big|^{p}\right]^{\frac{1}{p}}
&\lesssim_p
T^{\half} \const \donst^{-\frac12} (\log N)^\half \|\varphi\|_{1,2}\,,\label{e:nonlin:energy-estimate}\\
\Exp \left[ \sup_{t \le T} \Big|\int _0 ^t  \frac{\donst}{2} u^N(s,\Delta\varphi) \dd s\Big|^{p}\right]^{\frac{1}{p}}
&\lesssim_p
T^{\half}  \donst^{\half} \|\varphi\|_{1,2}\label{e:lin:energy-estimate}
\end{align}
where in both cases the implicit constant does not depend on $\varphi$, $T$ nor $N$. 
\end{lemma}


\begin{proof}
Let $\varphi$ be a test function in $H^1$ and recall the solutions $\pois(\varphi)$ and $\poisl(\varphi)$ of the Poisson equations 
defined in (and directly below of) \eqref{e:solPoiNL}, so that with the aid of Lemma~\ref{lem:ito-trick}
the proof of~\eqref{e:nonlin:energy-estimate} and~\eqref{e:lin:energy-estimate}
boils down to bounding the moments of $\energy(\pois(\varphi))(\eta)$ and $\energy(\poisl(\varphi))(\eta)$, 
with respect to the white noise measure. Since this measure is Gaussian and both $\pois(\varphi)$ and $\poisl(\varphi)$
live in a homogeneous Wiener chaos (of order $2$ and $1$ respectively), by Gaussian hypercontractivity
\cite[Theorem 3.50]{J} it suffices to bound the first moment of $\energy(\pois(\varphi))(\eta)$ and 
$\energy(\poisl(\varphi))(\eta)$. 
Let us begin with the former. In this proof, we adopt the convention that all sums over $\ell$ and $m$ are truncated by 
$\indN{\ell,m,\ell+m}$. Notice that 
\begin{equ}
\energy(\pois(\varphi))(\eta)=2\const^2\donst^{-2}\sum_{\ell\in\Z^2_0}|\ell|^2\Big|\sum_{m\in\Z^2_0} 
\nonlintilde_{\ell,m} \eta_m \varphi_{-\ell-m}\Big|^2
\end{equ}
where we set $ \nonlintilde_{\ell,m}\eqdef(|\ell|^2 + |m|^2)^{-1} \nonlin_{\ell,m} $. Upon taking expectation we get
\begin{equation}\label{e:QV}
\begin{split}
\E| \energy(\pois(\varphi))(\eta) |&=2\const^2\donst^{-2}\sum_{\ell,m\in\Z^2_0}|\ell|^2 |\nonlintilde_{l,m}|^2|\varphi_{-\ell-m}|^2\\
&=2\const^2\donst^{-2}\sum_{k\in\Z^2_0} |k|^2 \Big(\sum_{\ell\in\Z^2_0} \frac{|\ell|^2}{|k|^2}|\nonlintilde_{\ell,k-\ell}|^2\Big) 
|\varphi_k|^2
\end{split}
\end{equation}
where the first equality is a consequence of the fact that $\{\eta_m\}_m$ is 
a family of standard complex valued Gaussian random variables such that $\E[\eta_m\eta_\ell]=\mathds{1}_{\ell+m=0}$ for $l,m \ne 0$. In order to 
bound the quantity in the parenthesis, recall~\eqref{e:nonlinCoefficient} and set 
\begin{equation}\label{e:RSapprox}
f_k(z)\eqdef \frac{\thec(z, \tilde k-z)^2 }{ (|z|^2 + |\tilde k-z|^2)^2|\tilde k-z|^2 }
\,,\quad
\tilde k\eqdef \frac{k}{|k|},
\end{equation}
for $z \in \R^2$ and $k \in \Z^2_0$, and notice that $|f_k(z)| \lesssim g(z)\eqdef\frac{1}{|z|^2}\mathds{1}_{|z|>1} + 1 $ uniformly in $k \in \Z^2_0$ and $z \in \R^2$. 
Now, plugging in the definition of $\nonlin_{\ell,k-\ell}$, we have 
\begin{equ}
\sum_{\ell\in\Z^2_0} \frac{|\ell|^2}{|k|^2}|\nonlintilde_{\ell,k-\ell}|^2=\sum_{\ell\in\Z^2_0}\frac{1}{|k|^2} f_k\left(\frac{\ell}{|k|}\right)
\lesssim
\int_{|z| \le N} g(z) \dd z
\lesssim
\log N\,.
\end{equ}
Therefore we conclude that 
\begin{align}\label{e:energy-estimate}
\E\left[| \energy(\pois(\varphi))(\eta) |^p\right]^{\frac{1}{p}}\lesssim \const^{2} \donst^{-2} \log N \|\varphi\|^2_{1,2}
\end{align}
from which~\eqref{e:nonlin:energy-estimate} follows.  

The proof of~\eqref{e:lin:energy-estimate} is straightfoward since in this case the quadratic variation of 
$M_\cdot(\poisl(\varphi))$ is deterministic and we have $\energy(\poisl(\varphi))(\eta)=\|\varphi\|^2_{1,2}$. 
\end{proof}

\begin{remark}
At first sight, estimate~\eqref{e:lin:energy-estimate} might come as a surprise. Indeed, if we take $\donst\equiv 1$ it shows 
that {\it no matter how $\const$ behaves as $N\ua\infty$}, the bound would provide tightness for the sequence of 
approximations $\{u^N\}_N$ in a suitable space of {\it space-time distributions} (replace $\Delta\varphi$ with any $\psi$ smooth 
in~\eqref{e:lin:energy-estimate}). To understand this behaviour, consider, as an example, the family of SDEs
\begin{equ}
\dd X^N_t= -X^N_t + C_N
\left(\begin{matrix}
0 &1\\
-1 &0
\end{matrix}\right)
X^N_t \dd t+\dd B_t
\end{equ}
where $B$ is a two dimensional Brownian motion. 
Thanks to the It\^o trick, it is easy to see that the time average of $X^N$ 
stays uniformly bounded, {\it independently of the value of $C_N$}, thus giving tightness of $X^N$
in a space of distributions. That said, the time integral of $X^N$ represents a poor description of its actual behaviour
since, in case $C_N$ goes to $\infty$, $X^N$ is oscillating increasingly fast and the time integral simply converges to its average. 
%
%
\end{remark}

Lemma~\ref{lem:energy-estimate} suggests that, in order to control the non-linearity
in~\eqref{e:AKPZ:u} {\it uniformly} in $N$, we need to tune $\const$ and $\donst$ in such a way that the logarithmic factor 
on the right hand side of~\eqref{e:nonlin:energy-estimate} disappears. Let us define the integral in time of $\const\cN^N[u^N]$, as 
\begin{equation}\label{e:IntNonlin}
\inonlin_t[u^N](\varphi)\eqdef \int_0^t\const\cN^N[u^N(s)](\varphi)\dd s
\end{equation}
for any test function $\varphi\in H^1$. In the following theorem
we show that, under this scaling, the couple $\{(u^N, \inonlin[u^N])\}_N$ admits subsequential limits, 
in a (product) space of continuous functions in time with values in a space of distributions of suitable regularity.  

\begin{theorem}\label{thm:tightness}
Let $T>0$ and, for $N\in\N$, let $u^N$ be the stationary solution of~\eqref{e:AKPZ:u} and 
$\inonlin$ be the functional defined in~\eqref{e:IntNonlin}. Let $C>0$ and assume that $\const$ and 
$\donst$ satisfy 
\begin{equation}\label{e:ScalingRC}
\const\donst^{-\half}\sim \sqrt{\frac{ C}{\log N}}\,,\qquad \text{as $N$ tends to $\infty$.}
\end{equation}
Then, the sequence $\{(u^N,\inonlin[u^N])\}_N$ is tight in $C_T^\gamma\CC^\alpha\times C_T^\gamma\CC^\alpha$ 
for any $\gamma<1/2$ and $\alpha<-2$. 

Moreover, if $\donst=1$ for all $N\in\N$, then the sequence $\{(u^N,\inonlin[u^N])\}_N$ 
is tight in $C_T\CC^\alpha\times C_T\CC^\beta$ for any $\alpha<-1$ and $\beta<-2$.
\end{theorem}

\begin{remark}
It is not surprising that in case $\donst$ goes to $0$, we can prove tightness only in the same space where the space
time white noise lives. Indeed, although in this scenario the noise disappears in the limit, 
we also lose the smoothing effect of the Laplacian so that we cannot expect any regularisation coming from it. 
Since we are starting from a space white noise $\eta$ whose regularity is $-1-\eps$ 
for any $\eps>0$, then heuristically, power-counting suggests that the regularity of the nonlinearity 
(and consequently of the limit of $u^N$) should be $-2-2\eps$. 
\end{remark}

\begin{proof}
Choose sequences of coupling constants $\const$ and $\donst$ such that \eqref{e:ScalingRC} holds, 
let $u^N$ be the stationary solution of~\eqref{e:AKPZ:u}
and let $\inonlin$ be given by~\eqref{e:IntNonlin}. 
A natural way to establish tightness for a sequence of random processes is
Kolmogorov's criterion which, in the present context, requires a uniform control over the moments of the 
$\CC^\alpha\times\CC^\beta$-norm of the time increments of $(u^N,\inonlin[u^N])$. 
Thanks to the Markov property and the fact that $u^N$ is stationary for any $N\in\N$ we have
\begin{equs}
\Exp\big[\|u^N(t)-&u^N(r)\|_\alpha^p\big]=\Exp\left[\Exp\left[\|u^N(t)-u^N(r)\|_\alpha^p|\filt_r\right]\right]\\
&=\Exp\left[\Exp^{u^N(r)}\left[\|u^N(t-r)-u^N(0)\|_\alpha^p\right]\right]=\Exp\left[\|u^N(t-r)-u^N(0)\|_\alpha^p\right]
\end{equs} 
where $\{\filt_r\}_r$ is the filtration generated by $u^N$ and the previous holds for all $0\leq r<t\leq T$. An 
analogous computation can be carried out for $\inonlin[u^N]$ so that, for both, we can 
simply focus on the case $r=0$. 

Now, in order to obtain uniform bounds on $(u^N(t)-u^N(0), \inonlin_t[u^N])$ (clearly, $\inonlin_0[u^N]=0$) 
in a Besov space we need to understand 
the behaviour of their Littlewood-Paley blocks. For $\inonlin[u^N]$ we can immediately exploit Lemma~\ref{lem:energy-estimate}
and in particular~\eqref{e:nonlin:energy-estimate}. Indeed, it suffices to choose $\varphi$ to be 
$j$-th Littlewood-Paley kernel ($j\geq -1$)
so that 
\begin{equs}
\Exp\left[\|\Delta_j\inonlin_t[u^N]\|_{L^p(\T^2)}^p\right]&=\int_{\T^2} \E\left[\left|\inonlin_t[u^N](K_j(x-\cdot))\right|^p\right]\dd x\lesssim t^\frac{p}{2}\const^p\donst^{-\frac{p}{2}}(\log N)^{\frac{p}{2}}2^{2jp}
\end{equs}
where $K_j$ was defined above \eqref{def:Besov} and we used that $\|K_j\|^2_{1,2}\sim\sum_{|k|\sim 2^j} |k|^2\sim 2^{4j}$. Hence, by Besov embedding~\eqref{e:BesovEmb}
we have
\begin{equs}
\Exp\left[\|\inonlin_t[u^N]\|_{\alpha}^p\right]&\lesssim \Exp\left[\|\inonlin_t[u^N]\|_{B^{\alpha+d/p}_{p,p}}^p\right]=\sum_{j\geq -1} 2^{(\alpha+d/p)jp}\Exp\left[\|\Delta_j\inonlin_t[u^N]\|_{L^p(\T^2)}^p \right]\\
&\lesssim t^\frac{p}{2}\const^p\donst^{-\frac{p}{2}}(\log N)^{\frac{p}{2}} \sum_{j\geq -1} 2^{(\alpha+d/p+2)jp}
\end{equs}
and the latter sum converges if and only if $\alpha<-2-d/p$. Since the previous bound holds for any $p\geq 2$, by choosing
the renormalisation constants $\const$ and $\donst$ according to~\eqref{e:ScalingRC}, Kolmogorov implies that 
$\{\inonlin[u^N]\}_N$ is tight in $C_T^\gamma\CC^\alpha$ for any $\alpha<-2$ and $\gamma<1/2$. 
\newline

We now focus on $u^N$. By writing~\eqref{e:AKPZ:u} in its mild formulation and 
convolve both sides of the resulting expression with the $j$-th Littlewood-Paley kernel ($j\geq -1$), we obtain 
\begin{equation}\label{e:mildLPB}
\Delta_j(u^N(t)-\eta)=\Delta_j(P^N \eta(t)-\eta)+\const \Delta_jP^N \cN^N[u^N](t)+\donst^\half(-\Delta)^\half \Delta_jP^N\xi(t)
\end{equation}
where $P^N$ is the fundamental solution of $(\partial_t-\frac{\donst}{2}\Delta)P^N=0$ and, for any space-time distribution $f$,
$P^Nf$ denotes the space-time convolution between $P^N$ and $f$. 
At first we want to determine
bounds on the $p$-th moment of the $L^p$ norm of the three summands on the right hand side. 
For the first, using Gaussian hypercontractivity of $\eta$, we have
\begin{equs}
\E\Big[\|\Delta_j(P^N \eta(t)-&\eta)\|_{L^p(\T^2)}^p\Big]=\int_{\T^2} \E\left[\left|\Delta_j(P^N \eta(t)-\eta)\right|^p\right]\dd x\\
&\lesssim \int_{\T^2} \E\left[\left|\Delta_j(P^N \eta(t)-\eta)\right|^2\right]^\frac{p}{2}\dd x\\
&=\frac{1}{4\pi^2}\Big(\sum_{k\in\Z^2_0} \varrho_j(k)\left(e^{-\frac{\donst}{2}|k|^2t}-1\right)^2\Big)^{\frac{p}{2}}\lesssim t^{\frac{\kappa}{2} p} \donst^{\frac{\kappa}{2} p} 2^{jp(1+\kappa)}
\end{equs}
where the last bound is a consequence of the fact that $\varrho_j$ is supported on those $k\in\Z^2$ such that $|k|\sim 2^j$
and the geometric interpolation inequality, i.e. $1-e^{-\donst|k|^2t}\leq \min\{1,\donst|k|^2t\}\lesssim (\donst|k|^2t)^\kappa$, valid 
for any $\kappa\in[0,1]$ (applied above for $\tilde\kappa\eqdef\kappa/2$, $\kappa\in[0,2]$). 

To treat the second summand in~\eqref{e:mildLPB}, we want to rewrite it in such a way that Lemma~\ref{lem:ito-trick} is applicable. 
This is indeed possible since 
\begin{equ}
\const \Delta_jP^N \cN^N[u^N](t,x)=\int_0^t \gens \tilde H^N_{(t,x)}(s,u^N(s))\dd s
\end{equ}
where $\tilde H^N_{(t,x)}$ is the cylinder function depending smoothly on time defined by 
\begin{equ}
\tilde H^N_{(t,x)}(s,\eta)\eqdef \sum_{k\in\Z^2_0}e^{-\frac{\donst}{2}|k|^2(t-s)} \pois_k(\eta)\varrho_j(k) e_k(x)
\end{equ} 
and $\pois_k$ is the $k$-th Fourier component of the solution of the Poisson equation given in~\eqref{e:solPoiNL}.  
Hence, the $L^p$ norm of the Littlewood-Paley block is controlled by
\begin{equs}
\Exp\Big[\|\const \Delta_jP^N \cN^N&[u^N](t,\cdot)\|_{L^p(\T^2)}^p\Big]=\int_{\T^2} \Exp\Big[\Big|\int_0^t \gens \tilde H^N_{(t,x)}(s,u^N(s))\dd s\Big|^p\Big]\dd x\\
&\lesssim \donst^{\frac{p}{2}}
\left(\sum_{k\in \Z^2_0}\varrho_j(k)\Big(\frac{1-e^{-\donst|k|^2t}}{\donst|k|^2}\Big) \E\left[ | \energy ( \pois_k)(\swn) |^{\frac{p}{2}}\right]^{\frac{2}{p}}\right)^{\frac{p}{2}}\\
&\lesssim \donst^{\frac{p}{2}}
\left(\sum_{|k|\sim 2^j}\Big(\frac{1-e^{-\donst|k|^2t}}{\donst|k|^2}\Big) \const^{2} \donst^{-2} \log N |k|^2\right)^{\frac{p}{2}}\\
&\lesssim t^{\frac{\kappa}{2}p}\left(\const\donst^{-1+\frac{\kappa}{2}} (\log N)^\half\right)^p\, 2^{jp(1+\kappa)}
\end{equs}
where we went from the first to the second line via~\eqref{e:itoconvolution}, we subsequently bounded the $p/2$-moment 
of the energy through~\eqref{e:energy-estimate} and in the last line we used the same interpolation inequality as above. 

For the last term in~\eqref{e:mildLPB}, we can apply once more Gaussian hypercontractivity, but this time for the 
space time white noise $\xi$, to get 
\begin{equs}
\Exp\Big[\|&\donst^\half \Delta_jP^N \xi(t,\cdot)\|_{L^p(\T^2)}^p\Big]=
\int_{\T^2} \Exp\Big[\left|\donst^\half \Delta_jP^N \xi(t,x)\right|^p\Big]\dd x\\
&\lesssim \int_{\T^2} \Exp\Big[\left|\donst^\half \Delta_jP^N \xi(t,x)\right|^2\Big]^\frac{p}{2}\dd x=\left(\sum_{|k|\sim2^j} (1-e^{-\donst|k|^2t})\right)^{\frac{p}{2}}\lesssim t^{\frac{\kappa}{2}p} \donst^{\frac{\kappa}{2}p} 2^{jp(1+\kappa)}
\end{equs}
where in the last passage is again a consequence of the interpolation inequality.
\newline

Putting these three bounds together and applying Besov embedding~\eqref{e:BesovEmb}, we see that, for any $t> 0$ and 
$p\geq 2$, we have
\begin{equs}
\Exp\left[\|u^N(t)-u^N(0)\|_\alpha^p\right]&\lesssim \Exp\left[\|u^N(t)-u^N(0)\|_{B^{\alpha+d/2}_{p,p}}^p\right]\\
&=\sum_{j\geq -1}2^{(\alpha+d/p)j p}\Exp\left[\|\Delta_j(u^N(t)-u^N(0))\|_{L^p}^p\right]\\
&= 
t^{\frac{\kappa}{2}p} \donst^{\frac{\kappa}{2}p} \left(2+\left(\const\donst^{-1} (\log N)^\half\right)^p\right)
\sum_{j\geq -1}2^{j p (\alpha+d/p+1+\kappa)}\,.
\end{equs}
Now, notice that the last sum converges as soon as $\alpha<-1-\kappa-d/p$. Hence, 
if $\donst$ is a constant independent of $N$ and $\const\sim (\log N)^{-\half}$, 
we can conclude, by Kolmogorov's criterion, that the sequence $\{u^N\}_N$ is tight 
in the space $C_T\CC^\alpha$, with $\alpha$ arbitrarily close to (but strictly smaller than) $-1$. 
Otherwise, we are forced to choose $\kappa=1$ and the sequence 
$\{u^N\}_N$ is tight in $C_T^\gamma\CC^\alpha$ for all $\gamma<1/2$ and $\alpha<-2$. 
 \end{proof}
 
\begin{remark}
The previous theorem guarantees that if $\const$ and $\donst$ 
satisfy~\eqref{e:ScalingRC} then the couple $(u^N, \inonlin)$ converges (at least along a subsequence) to some limit $(u,\cB)$. In case that $\donst \to 0$ the energy estimate \eqref{e:lin:energy-estimate} of Lemma~\ref{lem:energy-estimate} implies that for any test function $\varphi\in H^1$ one has
\begin{equ}
u_t(\varphi)-u_0(\varphi)=\cB_t[u](\varphi)\,.
\end{equ}
Hence, a characterisation of the limit $u$ is connected to a deeper understanding of the process $\cB$. 
We are currently neither able to show that $\cB$ is $0$ nor are we able to define its law, so we leave its study to future 
investigations. 
\end{remark}
 
We define the integral in time of the nonlinearity for $h^N$, the solution of the approximation of the anisotropic KPZ equation 
in~\eqref{e:kpz:reg}, as
\begin{equation}\label{e:IntNonlinAKPZ}
\inonlintilde_t[h^N](\varphi)\eqdef \int_0^t\const\tilde\cN^N[h^N(s)](\varphi)\dd s\,
\end{equation}
where $\varphi$ is a generic test function and 
\begin{equation}\label{e:nonlin:KPZ}
\tilde \cN^N[h^N]\eqdef \Pi_N \Big((\Pi_N \partial_1 h^N)^2 - (\Pi_N \partial_2 h^N)^2\Big)\,,
\end{equation}
we can prove tightness for the sequence $\{(h^N,\inonlintilde[h^N])\}_N$.  
%
%
 
\begin{theorem}\label{thm:tightnessAKPZ}
Let $T>0$ and, for $N\in\N$, let $h^N$ be the solution of~\eqref{e:kpz:reg} started at $0$ from $\tilde\eta$, where for all 
$k\in\Z^2_0$, $\tilde\eta_k\eqdef |k|^{-1}\eta_k$, $\eta$ a space white noise, and $\tilde\eta_0=0$, and 
$\inonlintilde[h^N]$ be defined according to~\eqref{e:IntNonlinAKPZ}. 
Let $C>0$ and assume $\const$ and $\donst$ satisfy~\eqref{e:ScalingRC}. 

Then, the sequence $\{(h^N,\inonlintilde[h^N])\}_N$ is tight in $C_T^\gamma\CC^{\alpha+1}\times C_T^\gamma\CC^{\alpha+1}$ 
for any $\gamma<1/2$ and $\alpha<-2$. 
Moreover, if~\eqref{e:ScalingRC} is satisfied with $\donst$ a constant that is independent of $N$, 
then the sequence $\{(h^N,\inonlintilde[h^N])\}_N$ 
is tight in $C_T\CC^{\alpha+1}\times C_T\CC^{\beta+1}$ for any $\alpha<-1$ and $\beta<-2$.
\end{theorem}
\begin{proof}
For $\alpha\in\R$, define $\CC^\alpha_0$ as the set of functions in $\CC^\alpha$ whose $0$-th Fourier mode is $0$. Then, 
$\Delta^{1/2}$ is a homeomorphism between $\CC^\alpha_0$ and $\CC^{\alpha-1}_0$. Since by definition
$(u^N,\inonlin[u^N])=(\Delta^{1/2} h^N, \Delta^{1/2}\tilde\inonlin[h^N])$ and by Theorem~\ref{thm:tightness} $(u^N,\inonlin[u^N])$
is tight in $C_T^\gamma\CC^\alpha\times C_T^\gamma\CC^\alpha$ for any $\gamma<1/2$ and $\alpha<-2$ (resp. 
$C_T\CC^\alpha\times C_T\CC^\beta$ for any $\alpha<-1$ and $\beta<-2$, if $\donst$ constant) 
then the sequence $(h^N-h^N(e_0), \tilde\inonlin[h^N]-\tilde\inonlin[h^N](e_0))$ is tight in 
$C_T^\gamma\CC^{\alpha+1}\times C_T^\gamma\CC^{\alpha+1}$ (resp. $C_T\CC^{\alpha+1}\times C_T\CC^{\beta+1}$).
Therefore, it suffices to focus on $(h^N(e_0),\tilde\inonlin[h^N](e_0))$. Notice also that, since we chose $\tilde\eta_0=0$, 
$h^N_t(e_0)=\inonlintilde_t[h^N](e_0)$. 

In order to show tightness for the $0$-th Fourier mode of $h^N$ we want to apply again Lemma~\ref{lem:ito-trick}. To do so, 
we need to solve the Poisson equation $\gens \tilde H^N_0=\tilde\cN^N_0$. Proceeding as in~\eqref{e:solPoiNL}, we get
\begin{equation}\label{e:PoisAKPZ}
\tilde H^N_0[\eta]=2\const\donst^{-1}\sum_{\substack{\ell,m\in\Z^2_0\\\ell+m=0}}\frac{c(\ell,m)}{|\ell||m|}\eta_\ell\eta_m 
\end{equation}
and 
\begin{equs}
\E\left[\energy(\tilde H^N_0)(\eta) \right]= \sum_{k\in\Z^2_0}|k|^2 \E \left[|D_k\tilde H^N_0[\eta]|^2\right]
=4\const^2\donst^{-2}\sum_{|k|\leq N} \frac{c(k,-k)^2}{|k|^6}
\end{equs}
which, by~\eqref{e:ito-trick}, implies
\begin{equ}
\Exp[|h_t^N(e_0)|^p]^{\frac{1}{p}}=\Exp[|\tilde\inonlin[h^N](e_0)|^p]^{\frac{1}{p}}\lesssim t^{\frac{1}{2}}\const^2\donst^{-1}\log N
\end{equ}
and tightness follows. 
\end{proof}

\begin{remark}\label{rmk:NoRenom}
As opposed to the isotropic KPZ equation treated in~\cite{CSZ},~\cite{Gu} and \cite{CD19}, in the present context there is 
no average growth that needs to be subtracted in order to guarantee the convergence of the approximation. 
This is due do the fact that the nonlinearity in~\eqref{e:kpz:reg} has a further (anti-)symmetry with respect to change of variables
$\R^2\ni(x_1,x_2)\mapsto (x_2,x_1)$. 
\end{remark} 
%
%

\section{Lower bounds and non-triviality}
\label{S:nontrivialty}

Throughout this section, we will be assuming that for every $N\in\N$, $\donst=1$, so that the only renormalisation 
constant that we allow to vanish is $\const$. Notice that in this case the symmetric part of the generator $\gen$, $\gens$, 
does not depend on $N$ so we will simply denote it by $\gensy$

We aim at obtaining lower bounds on functionals of the solution $u^N$
to~\eqref{e:AKPZ:u} and to show that any subsequential limit $u$ is not trivial. By ``trivial'' here we mean that $u$ is  
the solution of the original equation without the nonlinearity, a scenario that could materialise in case 
$\const$ converges to $0$ too fast. 

To do so, we apply a technique, coming from particle systems (see~\cite{Landim2004}), which consists in determining
(and bounding) a variational formula for the Laplace transform of the integral in time of a suitable
functional of our process. We begin with the following lemma. 

\begin{lemma}\label{lem:Laplace}
Let $\{u^N(t)\}_{t\geq 0}$ be the stationary solution to~\eqref{e:AKPZ:u}
and $F\in L^2(\eta)$. 
Then, for every $\lambda>0$ the following equality holds
\begin{equation}\label{e:Laplace}
\int_0^\infty \dd t \,
	e^{-\lambda t}\Exp\left[ \left(\int_0^tF(u^N(s))\dd s\right)^2 \right] = 
\frac{2}{\lambda^2} \E\left[ F(\eta) (\lambda- \gen)^{-1} F(\eta)\right].
\end{equation}
\end{lemma}

\begin{proof}
Notice that we can rewrite the expectation at the left hand side of~\eqref{e:Laplace} as
\begin{equs}
\Exp\left[ \left(\int_0^tF(u^N(s))\dd s\right)^2 \right]
&=
2 \int_0^t \dd s \int_0^s \dd r \, \Exp\left[  F(u^N(r))F(u^N(s)) \right]
\\
&=
	2 \int_0^t \dd s \int_0^s \dd r \, \Exp[ F(u^N(r)) \Exp[ F(u^N(s)) | \filt_r] ],
\end{equs}
where $\filt$ denote the natural filtration of the process $\{u^N(t)\}_{t\geq 0}$. 
Now, $u^N$ is a Markov process, it generates a semi group which we denote by $\{e^{t\gen}\}_{t\geq 0}$, and at fixed time is distributed according to the 
law of $\eta$. Therefore, the right hand side of the previous is equal to
\begin{equs}
2 \int_0^t \dd s \int_0^s \dd r \, \E \Big[ F(\swn) &\Exp^{\swn}\left[F(u^N(s-r))\right]\Big]
=
2 \int_0^t \dd s \int_0^s \dd r \, \E \left[ F (\swn) e^{(s-r)\gen} F (\swn)\right]
\\
&=
2 \int_0^t \dd r \, (t-r) \E[ F (\swn)\, e^{r\gen} F (\swn)]\,.
\end{equs}
Here we use the symbol $\Exp^\swn$ to denote the expectation with respect to the law of the process 
$\{u^N(t)\}_{t\geq0}$ conditioned to start at $t=0$ from $\eta$.
Notice that the expectation in the last term above does not depend on $t$, hence its Laplace transform is given by 
\begin{equs}
2 \int_0^\infty \dd t  \int_0^t \dd r \, 
(t-r) &e^{-\lambda(t-r)}
\E\left[ F(\swn) e^{- r(\lambda - \gen)} F(\swn) \right]\\
&=
\frac{2}{\lambda^2} 
\E\Big[ F(\swn) \int_0^\infty dr \, e^{- r(\lambda - \gen)} F(\swn) \Big]
\end{equs}
where the equality is obtained by simply changing the order of integration. 
The conclusion now follows by applying the equality $\int_0^\infty dr \, e^{-r(\lambda-\gen)} = (\lambda-\gen)^{-1}$. 
\end{proof}

The advantage of the previous statement is twofold. At first,  notice that, while in principle the expectation at the left hand side 
of~\eqref{e:Laplace} depends on the distribution of the solution at different (at least $2$) points in time 
the right hand side only depends on the law of the invariant measure, which is explicitly known. 
Moreover, even though it is hard in general to invert the full generator (which is what seems to be required
in order to exploit Lemma~\ref{lem:Laplace}), the expression on the right hand side of~\eqref{e:Laplace} 
allows for a variational formulation which turns out to be easier to manipulate. 

This variational formula is given in~\cite[Theorem 4.1]{Komorowski2012} and, below, we 
state it in the way in which we will use it in the remainder of the section.

\begin{lemma}\label{lem:VarFor}
\emph{(Variational formula)} Let $\gen$ be the generator of the Markov process $\{u^N(t)\}_{t\geq 0}$ and let
$\gensy$ and $\gena$ defined in~\eqref{e:gens} and~\eqref{e:gena} be its symmetric and antisymmetric part with respect to the white noise measure $\swn$. Let $F\in L^2(\swn)$
and denote by $\langle\cdot,\cdot\rangle_\swn$ the scalar product in $L^2(\swn)$, then 
for every $\lambda>0$, 
\begin{equation}\label{e:VarFor}
\begin{split}
\langle F,(\lambda- \gen)^{-1} F\rangle_\swn=\sup_{G}\Big\{ 2\langle F,G\rangle_\swn -& \langle(\lambda-\gensy)G ,G\rangle_\swn\\
&-  \langle\gena G ,(\lambda-\gensy)^{-1}\gena G\rangle_\swn   \Big\}
\end{split}
\end{equation}
where $G$ ranges over a fixed core of $\gen$. 
\end{lemma} 
\begin{proof}
The lemma is a direct consequence of~\cite[Theorem 4.1]{Komorowski2012}. Indeed, it suffices to apply the first equality 
in the previous statement twice, so to simplify the term $\| A g\|_{-1,\lambda}^2$. 
\end{proof}

Thanks to the variational formula above, in order to obtain the lower bounds we 
are looking for, it suffices to find \emph{one} $G$ for which the quantity in brackets in~\eqref{e:VarFor} is bounded from below by a positive constant uniformly in $N$. 
The functional $F$ to which we will apply Lemmas \ref{lem:Laplace} and \ref{lem:VarFor} is the nonlinearity $\const\cN^N$, which, 
for fixed $N$, is a cylinder function belonging to a {\it fixed} (the second) Wiener chaos. 

Using the explicit expressions for $\gens$ and $\gena$, and 
the decomposition of $\gena$ from Lemma~\eqref{lem:generator:fockspace}, 
we are indeed able to determine such a function $G$ and consequently prove the following proposition. 

\begin{proposition}\label{p:LowerBound}
Let $\varphi\in H^1$ and $\cN^N$ be defined according to~\eqref{e:nonlin}. Let $C>0$ be such that 
\begin{equation}\label{e:ScalingRC1}
\const\sim\sqrt{\frac{C}{\log N}}\,,\qquad\text{as $N\to\infty$.}
\end{equation}
Then, there exists a constant $\delta>0$ independent of $N$ and $\varphi$ such that 
\begin{equation}\label{e:LowerBound}
\E\left[ \const\cN^N[\eta](\varphi)(\lambda- \gen)^{-1} \const\cN^N[\eta](\varphi)\right]\geq C\pi \delta\|\varphi\|_{1,2}^2
\end{equation}
for all $\lambda>0$. 
\end{proposition}
\begin{proof}
Firstly, we obtain a lower bound of the right hand side of ~\eqref{e:VarFor} by restricting to supremum to (smooth) random variables living in $\cH_2$,
the second homogeneous Wiener chaos of $\eta$.
With this choice, since, by Lemma~\ref{lem:generator:fockspace}, $\gena=\genap+\genam$ and $\genap$
maps $\cH_2$ to $\cH_3$ while $\genam$ maps $\cH_2$ to $\cH_1$, the quantity inside the brackets can be rewritten as
\begin{equation}\label{e:TermsBound}
2\langle \const\cN^N(\varphi), G\rangle-\langle(\lambda-\gensy)G ,G\rangle_\swn- \| (\lambda-\gensy)^{-\half}\genap G \|^2_{\swn} -
\| (\lambda-\gensy)^{-\half}\genam G \|^2_{\swn}
\end{equation}
where $\|\cdot\|_\swn\eqdef \|\cdot\|_{L^2(\swn)}$. Denote by $\one,\,\two,\,\three,\,\four$ each of the summands 
in~\eqref{e:TermsBound}, so that it equals $2\one-\two-\three-\four$.

Notice that, $\one$ is linear while the others are quadratic. In order to take advantage of this fact, 
it suffices to determine a function $G^N$, allowed to depend on $N$, such that, under the scaling~\eqref{e:ScalingRC1}, $\two$-$\four$ are bounded uniformly, while $\one$ is bounded uniformly from below by a positive constant. 

Hence, let $\delta>0$ be a constant to be fixed later and take $G$ to be the solution of the Poisson equation~\eqref{e:solPoiNL}
with $\donst=1$, multiplied by $\delta$, i.e. $G(\eta)\eqdef \delta \pois[\eta](\varphi)$. Then, we have
\begin{equ}
\one= \langle \const\cN^N(\varphi), \delta\pois(\varphi)\rangle_\eta= 
4 \delta \sum_{k\in\Z^2_0}|k|^2\left(\const^2\sum_{\ell+ m=k} \frac{1}{|k|^2}\frac{(\nonlin)_{\ell,m}^2}{|\ell|^2+|m|^2}\right)|\varphi_{-k}|^2\,.
\end{equ}
Now, the quantity in bracket can be analysed with the same tools used in the proof of~\eqref{e:LimitQV}, so we 
address the reader to the section below for the details and, here, limit 
ourselves to outline the procedure and highlight the main steps. By a Riemann sum
approximation, we have 
\begin{equs}
\const^2&\sum_{\ell+ m=k} \frac{1}{|k|^2}\frac{(\nonlin)_{\ell,m}^2}{|\ell|^2+|m|^2}=
\const^2\sum_{|\ell|,|k-\ell|\leq N} \frac{c(\ell,k-\ell)}{|\ell|^2|k-\ell|^2(|\ell|^2+|k-\ell|^2)}\\
&=\const^2\sum_{|\ell|,|k-\ell|\leq N} \frac{1}{|k|^2}\tilde f_k\left(\frac{\ell}{|k|}\right)\approx\const^2
\int_{\frac{5}{2}\leq|x|\leq \frac{N}{|k|}}\frac{c(x,\tilde k-x)^2}{|x|^2|\tilde k-x|^2(|x|^2+|\tilde k-x|^2)}\dd x
\end{equs}
where $\tilde f_k(x)$ coincides with the integrand on the last term of the previous equality and $\tilde k= k/|k|$. Since $\tilde k$
has norm one, let $\theta_k\in[0,2\pi)$ be such that $\tilde k= (\cos \theta_k,\sin\theta_k)$. Then, passing to polar coordinates
and neglecting all terms in the integral which are uniformly bounded in $N$ (since they are then 
killed by the vanishing constant $\const$), the previous is approximated by 
\begin{equs}
\frac{\const^2}{4}&\int_0^{2\pi} \int_{\frac{5}{2}}^{\frac{N}{|k|}}\frac{r\cos^2(2\theta)}{(r-\cos(\theta-\theta_k))^2} \dd r\dd \theta\\
&\approx 
\frac{1}{4} \int_0^{2\pi} \cos^2(2\theta) \left(\const^2 \log\left(\frac{N/|k|-\cos(\theta-\theta_k)}{5/2-\cos(\theta-\theta_k)}\right)\right)
\dd \theta\overset{N\to\infty}{\longrightarrow}C\,\frac{\pi}{4}\,.
\end{equs}
which implies that as $N$ goes to $\infty$, 
\begin{equation}\label{b:one}
\one\sim \delta C_\one \|\varphi\|_{1,2}^2
\end{equation}
where $C_\one\eqdef C\pi$. For $\two$, notice that, by definition of $\pois(\varphi)$, as $N\to\infty$ we have
\begin{equation}\label{b:two}
\two= \delta^2 \langle (\lambda-\gensy) \pois(\varphi), \pois(\varphi)\rangle_\eta=\lambda\delta^2 \| \pois(\varphi)\|_\eta^2-\delta \one\sim -\delta^2 C_\one \|\varphi\|_{1,2}^2
\end{equation}
locally uniformly in $\lambda$. The last passage will be justified in detail in the proof of Corollary~\ref{cor:Cherry} 
where we will see that the $L^2(\eta)$-norm of $\pois(\varphi)$ converges to $0$ as $N\to\infty$ (see~\eqref{e:Pois0}). 

We can now focus on $\three$. By Lemma~\ref{lem:generator:fockspace}, the Fourier transform of the kernel (in Fock space
representation) of $(\lambda-\gensy)^{-\half}\genap \pois(\varphi)$ is given by 
\begin{equ}
\cF\left( (\lambda-\gensyf)^{-\half}\genapf\mathfrak{H}^N_\varphi\right)(\ell,m,n)= 
4\const^2 \frac{\nonlin_{\ell,m}\nonlin_{\ell+m,n}\, \varphi_{-\ell-m-n}}{(\lambda +\frac{1}{2}(|\ell|^2+|m|^2+|n|^2)^\half(|\ell+m|^2+|n|^2)}
\end{equ}
where we denoted by $\mathfrak{H}^N_\varphi$ the kernel of $\pois(\varphi)$. Hence, we get 
\begin{equs}
\|(\lambda&-\gensy)^{-\half}\genap \pois(\varphi)\|_{\swn}^2\\
&\lesssim  \sum_{k\in\Z^2_0}\left(\const^4\sum_{\ell+m+n=k}\frac{(\nonlin_{\ell,m}\nonlin_{\ell+m,n})^2}{(\lambda +\frac{1}{2}(|\ell|^2+|m|^2+|n|^2))(|\ell+m|^2+|n|^2)^2}\right)|\varphi_{-k}|^2\\
\end{equs}
By bounding brutally $(\nonlin_{\ell,m}\nonlin_{\ell+m,n})^2\leq |k|^2 |n|^{-2}$, the quantity in brackets above can be 
treated as
\begin{equs}
\const^4&\sum_{\ell+m+n=k}\frac{(\nonlin_{\ell,m}\nonlin_{\ell+m,n})^2}{(\lambda +\frac{1}{2}(|\ell|^2+|m|^2+|n|^2))(|\ell+m|^2+|n|^2)^2}\\
&\leq |k|^2 \const^2\sum_{|n|\leq N}\frac{1}{|n|^2}\const^2\sum_{|\ell|\leq N}\frac{1}{\lambda +\frac{1}{2}(|\ell|^2+|m|^2+|n|^2)}
\leq |k|^2 \left(\const^2\sum_{|\ell|\leq N}\frac{1}{|\ell|^2}\right)^2
\end{equs}
and the choice of $\const$ guarantees that the previous is uniformly bounded by $|k|^2$. Therefore, there exists a constant
$C_\three>0$ independent of $N,\lambda,\varphi$ such that 
\begin{equation}\label{b:three}
\three\leq \delta^2 C_\three \|\varphi\|_{1,2}^2\,.
\end{equation}
It remains to study $\four$. Again by Lemma~\ref{lem:generator:fockspace}, the Fourier transform of the kernel (in Fock space
representation) of $\genam \pois(\varphi)$ is given by 
\begin{equs}
\cF\left( \genamf\mathfrak{H}^N_\varphi\right)(k)=
4\const  \sum_{\ell+m=k} \nonlin_{k,-\ell}\cF(\mathfrak{H}^N_\varphi)(\ell,m)
=8\const^2\sum_{\ell+m=k} \frac{c(k,-\ell)c(\ell,m)}{|\ell|^2(|\ell|^2+|m|^2)}\varphi_{-k}\,.
\end{equs}
Let us observe the inner sum more carefully. Define $K(\ell,m)\eqdef c(\ell,m)(|\ell|^2+|m|^2)^{-1}$, which is clearly 
symmetric in $\ell$ and $m$. Then, by changing variables in the sum ($\ell\to k-\ell$) and using the fact that $c$ is a 
symmetric bilinear form in its arguments (the antisymmetry is only by swapping the coordinates of both variables), so that 
in particular $c(\ell,-m)=-c(\ell,m)$, we have
\begin{equs}
&\sum_{\ell+m=k} \frac{c(k,-\ell)}{|\ell|^2}K(\ell,k-\ell)=
\half\sum_{|\ell|,|k-\ell|\leq N}\left( \frac{c(k,-\ell)}{|\ell|^2}-\frac{c(k,k-\ell)}{|k-\ell|^2} \right) K(\ell,k-\ell)\\
&=-\frac{c(k,k)}{2}\sum_{|\ell|,|k-\ell|\leq N}\frac{K(\ell,k-\ell)}{|k-\ell|^2}+\half\sum_{|\ell|,|k-\ell|\leq N}c(k,\ell)\left( \frac{1}{|k-\ell|^2}-\frac{1}{|\ell|^2} \right) K(\ell,k-\ell)\,.
\end{equs} 
Now, since for any $\ell,m$, $|K(\ell,m)|\lesssim 1$ and $|c(\ell,m)|\leq |\ell||m|$, it is immediate to see 
that the first summand is bounded by $|k|^2\log N$. For the second, by Taylor formula (holding at least for $|\ell|$ large enough, 
say $|\ell|>2|k|$) we have
\begin{equ}
\left|\sum_{|\ell|,|k-\ell|\leq N}c(k,\ell)\left( \frac{1}{|k-\ell|^2}-\frac{1}{|\ell|^2} \right) K(\ell,k-\ell)\right|
\lesssim |k|^2\sum_{|\ell|\leq N} \frac{1}{|\ell|^2}\lesssim |k|^2 \log N\,.
\end{equ}
Exploiting~\eqref{e:ScalingRC1} to get rid of the $\log$ divergence, it follows that there exists a constant $C_\four>0$ for which 
\begin{equ}
\|(\lambda-\gensy)^{-\half}\genam \pois(\varphi)\|_\swn^2=\sum_{k\in\Z^2_0}\frac{|\cF\left( \genamf\mathfrak{H}^N_\varphi\right)(k)|^2}{\lambda+|k|^2}\leq C_\four\sum_{k\in\Z^2_0}\frac{|k|^4|\varphi_{-k}|^2}{\lambda+|k|^2}\leq C_\four \|\varphi\|_{1,2}^2
\end{equ}
and consequently
\begin{equation}\label{b:four}
\four\leq \delta^2C_\four\|\varphi\|_{1,2}^2\,.
\end{equation}
Collecting~\eqref{b:one},~\eqref{b:two},~\eqref{b:three},~\eqref{b:four}, we see that~\eqref{e:TermsBound} is bounded below by 
$\delta{C_\one}(2-\delta \tilde C)\|\varphi\|_{1,2}^2$, where $\tilde C\eqdef 2 C_\one^{-1}\max\{C_\one,C_\three,C_\four\}-1$, 
and therefore, for any $\delta\in (0, 2\tilde C^{-1})$,~\eqref{e:LowerBound} holds. 
\end{proof}

As an immediate consequence of the previous proposition, we can show that the 
Laplace transform of the integral in time of the (rescaled) nonlinearity of~\eqref{e:AKPZ:u} is (uniformly) bounded 
from above and below. 

\begin{corollary}\label{cor:NontrivNonlin}
For any $N\in\N$, let $u^N$ be the solution of~\eqref{e:AKPZ:u} and $\cN^N$ be the functional defined by~\eqref{e:nonlin}. 
Assume $\donst=1$ for all $N$ and the sequence of positive constants 
$\const$ satisfies the scaling relation~\eqref{e:ScalingRC1}. Then, 
there exists a constant $\delta>0$ such that for any $\lambda>0$, $\varphi\in H^1$ and $N\in\N$ we have
\begin{equation}\label{e:L-Ubound}
\frac{\delta}{\lambda^2}\|\varphi\|_{1,2}^2\leq \int_0^\infty e^{-\lambda t} \Exp\left[ \left(\int_0^t\const \cN^N[u^N(s)](\varphi)\dd s\right)^2 \right]\dd t\leq \frac{\delta^{-1}}{\lambda^2}\|\varphi\|_{1,2}^2\,.
\end{equation}
\end{corollary}
\begin{proof}
The proof of the lower bound in~\eqref{e:L-Ubound} is a direct consequence of Lemma~\ref{lem:Laplace} and 
Proposition~\ref{p:LowerBound}. The upper bound instead follows by~\eqref{e:nonlin:energy-estimate} in 
Lemma~\ref{lem:energy-estimate}, upon taking $p=2$ and evaluating the Laplace transform of $f(t)=t$. 
%
\end{proof}

In the following proposition we collect the results obtained so far and provide a description of the limit points of the 
sequence $u^N$. 

\begin{proposition}\label{p:LimitPoints}
For $N\in\N$, let $u^N$ be the stationary solution of~\eqref{e:AKPZ:u}.
Assume $\donst=1$ and the sequence of constants $\const$ satisfies~\eqref{e:ScalingRC1}. Then,  
any subsequential limit $(u,\inonlin[u])$ of $\{(u^N,\inonlin[u^N])\}_N$ satisfies
\begin{equation}\label{e:LimitPoints}
u_t(\varphi)-u_0(\varphi)=\frac{1}{2}\int_0^t u_s(\Delta\varphi)\dd s+\cB_t[u](\varphi)+B_t(\varphi)
\end{equation}
for any $\varphi\in H^1$, and $\cB_t[u](\varphi)$ is a stationary stochastic process such that 
\begin{equation}\label{e:smalltime}
\Exp[\cB_t[u](\varphi)^2]\sim t\,,\qquad\text{as $t\to 0$. }
\end{equation}
In particular, it has non-zero finite energy. 
\end{proposition}
\begin{proof}
The validity of~\eqref{e:LimitPoints} is a consequence of Theorem~\eqref{thm:tightness}, hence the only thing to prove 
is that the process $\{\cB_t[u](\varphi)\}_{t\geq 0}$ satisfies~\eqref{e:smalltime}. Once the latter is established, 
we can immediately conclude that the process has non-zero energy, and the fact that it is finite follows 
by~\eqref{e:nonlin:energy-estimate}. 

Now, for~\eqref{e:smalltime}, we need the estimate~\eqref{e:L-Ubound}, which clearly holds also for $\cB_\cdot(\varphi)$, i.e.
\begin{equ}
\frac{\delta}{\lambda^2}\|\varphi\|_{1,2}^2\leq \int_0^\infty e^{-\lambda t} \Exp\left[ \cB_t(\varphi)^2 \right]\dd t\leq \frac{\delta^{-1}}{\lambda^2}\|\varphi\|_{1,2}^2\,,
\end{equ}
and~\eqref{e:nonlin:energy-estimate}. Lemma~\ref{lem:ST}, whose proof is provided in Appendix~\ref{ap:HLTT}, 
allows to conclude.  
\end{proof}

\begin{remark}
The previous proposition marks the difference between the $1$ and the $2$ dimensional case. Indeed, for 
$d=1$,~\cite{GubinelliJara2012}
shows that the solution of the KPZ equation (or stochastic Burgers) is a Dirichlet process, i.e. the sum of a martingale 
and a zero quadratic variation process. In particular, the integral in time of the nonlinearity converges to a $0$-quadratic 
variation process. In the two dimensional anisotropic case instead, the relation~\eqref{e:smalltime} suggests that 
the integral in time of the nonlinearity should morally 
contain a martingale part (hence in particular, if it admitted quadratic variation, it would be non zero) 
whose understanding would represent the main step in the characterisation of the limit points. 
\end{remark}

We conclude this section by stating (and proving) an analogous result at the level of the anisotropic KPZ equation. 
In this context we show that, assuming the noise to have zero average, the time increment of the average of the solution 
does not vanish, thus distinguishing it from the solution of the stochastic heat equation $X$ defined in~\eqref{e:SHE}. 

\begin{theorem}\label{thm:AKPZ-nontriviality}
For $N\in \N$, let $h^N$ be the solution of the smoothened anisotropic KPZ equation~\eqref{e:kpz:reg} and $\tilde X$ be 
the solution of the stochastic heat equation obtained by setting $\const=0$ in~\eqref{e:kpz:reg}, both 
started at $0$ from $\tilde\eta$, defined as in the statement of Theorem~\ref{thm:tightnessAKPZ}. Assume 
that the constants $\donst$ and $\const$ are such 
that $\donst=1$ and $\const$ satisfies~\eqref{e:ScalingRC1}. 

Then, any limit point $\{h(t,\cdot)\}_{t\geq 0}$ of the sequence $\{h^N\}_N$, is a stochastic processes different in law from 
$\{\tilde X(t,\cdot)\}_{t\geq 0}$.
\end{theorem} 
\begin{proof}
Let $h$ be a limit point of the sequence $\{h^N\}_{N\in\N}$ and $\tilde X$ be the solution of the stochastic heat equation. 
In order to prove the statement, it suffices to exhibit {\it any} observable which is different for $h$ and $\tilde X$. 
An observable easy to treat is the $0$-th Fourier mode $h_0$ and $\tilde X_0$ of $h$ and $X$, i.e. their spatial average. 
Notice that by construction $\tilde X_0=0$, while, by \eqref{e:kpz:reg} one has $h_0^N(t) = \const \tilde\CN^N_0[h^N(s)] ds$, 
so that it suffices to show that there exists $\delta>0$ (a priori depending on $t$) for which
\begin{equation}\label{e:0FM}
\Exp\left[\left( \int_0^t\const \tilde\cN^N_0[h^N(s)]\dd s\right)^2 \right]>\delta\,.
\end{equation}
For this we will exploit the same strategy 
as in the proof of Proposition~\ref{p:LowerBound} and Corollary~\ref{cor:NontrivNonlin}.
To be more precise, we consider the 
Laplace transform of $\Exp[(\int_0^t\const \tilde\cN^N_0[h^N(s)]\dd s)^2]$, to which we apply Lemmas~\ref{lem:Laplace} 
and~\ref{lem:VarFor}. In the variational problem, we take $G$ to be $\theta \tilde H^N_0$, 
where $\theta$ is a positive constant and $\tilde H^N_0$ is the solution 
of the Poisson equation $\gensy \tilde H^N_0(\eta)=\const \cN^N[\eta]$ obtained in~\eqref{e:PoisAKPZ}. 
We now need to control the four terms in the brackets of the right hand side of~\eqref{e:VarFor}.  
We treat the second summand as in~\eqref{b:two}
and, since $\genam\tilde H^N_0=0$, we are left to consider the first and the third, which give
\begin{equs}
\langle \const\tilde \cN^N_0,\tilde H^N_0\rangle&=\theta\const^2\sum_{|\ell|\leq N} \frac{c(\ell,\ell)^2}{|\ell|^6}\overset{N\to\infty}{\longrightarrow} \theta C\pi\\ 
&\text{while}\\
\| (\lambda-\gensy)^{-\half} \genap\tilde H^N_0\|^2_{L^2(\swn)}&=\theta^2\const^4 \sum_{\ell+m=n}\frac{c(n,n)^2}{|n|^6}
\frac{c(\ell,m)^2}{(\lambda +\half(|\ell|^2+|m|^2+|n|^2)|m|^2|\ell|^2}\\
&\lesssim \theta^2\const^4 \left(\sum_{|n|\leq N}\frac{1}{|n|^2}\right)^2\lesssim \theta^2\,.
\end{equs}  
Hence, following the same steps as in the proof of Proposition~\ref{p:LowerBound}, we conclude that there exists $t>0$ and
$\delta>0$ (a priori depending on $t$) for which~\eqref{e:0FM} holds and the proof is concluded.
\end{proof}

\section{Further consequences of the It\^o trick}\label{sec:Consequence}

In this section,
we want to make some further observation on the martingales appearing on the right hand side of~\eqref{eq:Ito}. 
This is done in order to shed some light on the behaviour we might expect for these limit points and could
represent a starting point for their characterisation. 

We want to analyse the martingale associated to the solution of the Poisson equation~\eqref{e:Poisson}
for $V$ given by $\const\cN^N[\eta](\varphi)$, where $\varphi$ is some test function, say $\varphi\in H^1$.
We define
\begin{equation}\label{def:Mart}
\martnl_t(\varphi)\eqdef \donst^\half \martf_t(\pois[u^N](\varphi))
\end{equation}
where the definition of the martingale $\martf$ on the right hand side and $\pois(\varphi)$ can be found in~\eqref{e:Mart} 
and~\eqref{e:solPoiNL} respectively. 
In the following proposition, we show that, upon choosing the renormalising constants in such a way that 
the right hand side of~\eqref{e:nonlin:energy-estimate} is uniformly bounded, in the limit as $N\to\infty$, 
$\martnl_t(\varphi)$ converges to a Brownian motion. 

\begin{proposition}\label{p:MartConv}
Let $\varphi\in H^1$ and, for any $N\in\N$, $\{\martnl_t(\varphi)\}_{t\geq 0}$ be the martingale defined in~\eqref{def:Mart}. 
Let $C>0$ be a real constant for which~\eqref{e:ScalingRC} holds. Then, the sequence of martingales $\{\martnl(\varphi)\}_N$ converges in distribution to a Brownian motion whose 
quadratic variation is given by $t Q(\varphi)$, where $Q(\varphi)$ is defined as
\begin{equation}\label{e:LimitQV}
Q(\varphi)= 4\,C\,\pi\,\|\varphi\|_{1,2}^2\,.
\end{equation} 
\end{proposition}
\begin{proof}
According to~\cite[Theorem 7.1.4]{EK}, since for every $N$ the martingale $\martnl(\varphi)$ is continuous, the proof of
the statement follows once we show that its quadratic variation converges in probability to a deterministic function of time. 
Now, the quadratic variation of $\martnl(\varphi)$ is explicit and can be deduced by~\eqref{e:energy}. 
The choice of the renormalisation
constants in~\eqref{e:ScalingRC} and~\eqref{e:energy-estimate} imply that $\langle \martnl_\cdot(\varphi)\rangle$
has bounded moments of all orders so we are left to prove that its variance vanishes in the limit $N\to\infty$ 
and show~\eqref{e:LimitQV}. 
Notice that, by~\eqref{e:QV}, we have 
\begin{equ}
\Exp\left[\langle \martnl_\cdot(\varphi)\rangle_t\right]=t \donst\E\left[\energy(\pois(\varphi))(\eta)\right]
\end{equ}
and 
\begin{equ}
\Exp\left[\langle \martnl_\cdot(\varphi)\rangle_t^2\right]^\frac{1}{2}=\donst\Exp\left[\left(\int_0^t\energy(\pois[u^N(s)](\varphi))\dd s\right)^2\right]^\frac{1}{2}\lesssim t\,\donst \E\left[\energy(\pois(\varphi))(\eta)^2\right]^\frac{1}{2}
\end{equ}
and we can compute the last expectation explicitly. Wick's theorem for the product of 
Gaussian random variables~\cite[Theorem 1.36]{J} gives
\begin{equ}
\donst^2\E\left[\energy(\pois(\varphi))(\eta)^2\right]=
\donst^2
\Big(
	\E\left[\energy(\pois(\varphi))(\eta)\right]
\Big)^2 +R_N(\varphi)
\end{equ}
where the remainder $R_N(\varphi)$ is given by 
\begin{equ}
R_N(\varphi)\eqdef 4^4\const^4\donst^{-2}
\sum_{\substack{\ell_1,\ell_2\in\Z^2_0\\k_1,\dots,k_4\in\Z^2_0}}\Big(\prod_{\substack{i=1,2\\j=1,\dots,4}}
|\ell_i|\nonlintilde_{\ell_i,k_j-\ell_i}\Big)\varphi_{k_1}\varphi_{-k_2}\varphi_{k_3}\varphi_{-k_4} 
\left(\1_{A}+\1_{B}\right)
\end{equ}
and, to shorten the notation, we set $\nonlintilde_{\ell,m}\eqdef (|\ell|^2+|m|^2)^{-1}\nonlin_{\ell,m}$, $\nonlin$ 
as in~\eqref{e:nonlinCoefficient},  and the two sets appearing at the right hand side are 
$A=\{\ell_1,\ell_2,k_1,\dots,k_4\in\Z^2_0\,:\,k_1-\ell_1+k_3-\ell_2=0=\ell_1-k_2+k_4-\ell_2\}$
and $B=\{\ell_1,\ell_2,k_1,\dots,k_4\in\Z^2_0\,:\,k_1-\ell_1+\ell_2-k_4=0=\ell_1-k_2+k_3-\ell_2\}$. 
The two terms can the treated similarly, so we will focus on the second. 
Brutally bounding $|\nonlintilde_{\ell_i,k_j-\ell_i}|\lesssim|k_j|(|\ell_i||k_j-\ell_i|)^{-1}$ and using that, 
when restricted to $B$, we can express $\ell_2$ and $k_4$ in terms of $\ell_1$ 
and $k_1,k_2,k_3$, respectively, we can estimate the above sum by 
\begin{equs}
\const^4\donst^{-2}&\sum_{\substack{k_1,k_2,k_3\in\Z_0^2\\k_4=k_3+k_2-k_1}}\prod_{i=1}^4 |k_i| 
\left(\sum_{\ell\in\Z^2_0} \frac{1}{|k_1-k_2+\ell|^2|\ell|^2}\right)\varphi_{k_1}\varphi_{-k_2}\varphi_{k_3}\varphi_{-k_4} \\
&\lesssim \const^4\donst^{-2}\sum_{\substack{k_1,k_2,k_3\in\Z_0^2\\k_4=k_3+k_2-k_1}}\prod_{i=1}^4 |k_i| 
\varphi_{k_1}\varphi_{-k_2}\varphi_{k_3}\varphi_{-k_4}\lesssim \const^4\donst^{-2}\|\varphi\|^4_{1,2}
\end{equs}
where the passage from the first to the second line is due to the fact that the inner sum converges, while the second is a 
consequence of Young's convolution inequality. 

Therefore, collecting the observations made so far we have
\begin{equ}
\mathrm{Var}\left[\langle \martnl_\cdot(\varphi)\rangle_t\right]=\Exp\left[\langle \martnl_\cdot(\varphi)\rangle_t^2\right]-
\Exp\left[\langle \martnl_\cdot(\varphi)\rangle_t\right]\lesssim t^2 \const^4\donst^{-2}\|\varphi\|^4_{1,2}
\end{equ}
Now, by~\eqref{e:ScalingRC}, the right hand side converges to $0$ as $N$ tends to $\infty$, 
for every $t>0$. In particular, by dominated convergence, this means that $\langle \martnl_\cdot(\varphi)\rangle_t$ 
converges to the limit of its expectation (i.e. to a deterministic function of time). Therefore, it remains to 
identify $\lim_N \Exp [\langle \martnl_\cdot(\varphi)\rangle_t]$ for which 
we need to refine the estimates in the proof of Lemma~\ref{lem:energy-estimate}. 
In~\eqref{e:QV}, we showed the following identity
\begin{equs}
\donst\E\left[\energy(\pois(\varphi))(\eta)\right]=
4^2 \sum_{k\in\Z^2_0} |k|^2 \Big(\const^2\donst^{-1}\sum_{\ell\in\Z^2_0} \frac{|\ell|^2}{|k|^2}|\nonlintilde_{\ell,k-\ell}|^2\Big) 
|\varphi_k|^2
\end{equs}
and the part to control is the one in bracket. By Riemann-sum approximation, we can rewrite the latter (for $k\in\Z^2_0$
fixed such that $|k|\leq N$) as
\begin{equs}
\const^2\donst^{-1}&\sum_{\ell\in\Z^2_0} \frac{|\ell|^2}{|k|^2}|\nonlintilde_{\ell,k-\ell}|^2=
\const^2\donst^{-1}\sum_{|\ell|,|k-\ell|\leq N}\frac{c(\ell,k-\ell)^2}{|\ell|^2(|\ell|^2+|k-\ell|^2)}\\
&=\const^2\donst^{-1}\sum_{|\ell|,|k-\ell|\leq N} \frac{1}{|k|^2}f_k\left(\frac{\ell}{|k|}\right)
\approx \const^2\donst^{-1}\int_{\frac{3}{2}\leq|x|\leq \frac{N}{|k|}} \frac{c(x,\tilde k-x)^2}{|x|^2(|x|^2+|\tilde k-x|^2)^2}\dd x
\end{equs}
where $f_k$ was defined in~\eqref{e:RSapprox} and $\tilde k=k/|k|$. Since $\tilde k$ has Euclidean norm $1$, 
let $\theta_k\in[0,2\pi)$ be such that $\tilde k=(\cos \theta_k,\sin \theta_k)$, so that, by passing to polar coordinates
(and exploiting basic trigonometric identities) the integral becomes
\begin{equs}
\const^2\donst^{-1}\int_0^{2\pi}\int_{\frac{3}{2}}^{\frac{N}{|k|}}&\frac{(r\cos(2\theta) -\cos (\theta+\theta_k))^2}{(2 r^2-2r\cos(\theta-\theta_k)+1)^2}\,r\,\dd r\,\dd \theta\\
&\approx \const^2\donst^{-1}\int_0^{2\pi}\int_{\frac{3}{2}}^{\frac{N}{|k|}}\frac{r\cos^2 (2\theta) }{4(r-\cos(\theta-\theta_k))^2}\,\dd r\,\dd \theta
\end{equs}
and the last approximation holds since the integrals in which at the numerator $r$ is raised to a power smaller 
than $2$ are uniformly bounded in $N$ and therefore they converge to $0$ because of the prefactor $\const^2\donst^{-1}$. 
Now, adding and subtracting $\cos(\theta-\theta_k)\cos^2 (2\theta)$ at the numerator, and arguing as above, 
we can further approximate the quantity above by
\begin{equs}
\approx \const^2\donst^{-1}&\int_0^{2\pi}\int_{\frac{3}{2}}^{\frac{N}{|k|}}\frac{\cos^2 (2\theta) }{4(r-\cos(\theta-\theta_k))}\,\dd r\,\dd \theta\\
&=\frac{1}{4}\int_0^{2\pi} \cos^2 (2\theta) \left(\const^2\donst^{-1}\log\left(\frac{N/|k|- \cos(\theta-\theta_k)}{3/2- \cos(\theta-\theta_k)}\right)\right) \dd \theta\overset{N\to\infty}{\longrightarrow} C\, \frac{\pi}{4}
\end{equs}
where in the last passage we used dominated convergence theorem and the proportionality constant $C>0$
in~\eqref{e:ScalingRC}.
Hence, we conclude that 
\begin{equ}
\Exp\left[\langle \martnl_\cdot(\varphi)\rangle_t\right]\overset{N\to\infty}{\longrightarrow} 4\,\pi\,C\, t\, \|\varphi\|_{1,2}^2
\end{equ}
which completes the proof. 
\end{proof}

\begin{remark}
Notice that, the previous proposition allows to understand the behaviour of the two martingales on the right hand 
side of~\eqref{e:Ito}. In order to obtain a characterisation of the limit points of the nonlinearity we would need to understand the
joint correlation between them. The problem is not easy and out of reach of the techniques of the present paper. 
\end{remark}

As an easy corollary of the previous proposition, we show how to construct the time average of nonlinear unbounded 
functionals of the solution to the stochastic heat equation, purely by martingale techniques. The 
stochastic heat equation
we have in mind, is the stochastic PDE whose expression is given by
\begin{equation}\label{e:SHE}
\partial_t X=\frac{1}{2}\Delta X +(-\Delta)^{\half}\xi\,,\qquad X(0,\cdot)=\eta
\end{equation}
where $\xi$ and $\eta$ are respectively a space time and a space white noise on $\T^2$. Existence and uniqueness
of probabilistically strong solution is well-known and a martingale characterisation can be found, e.g., in~\cite[Appendix D]{MW}. 
We are now ready to state and prove the following. 

\begin{corollary}\label{cor:Cherry}
Let $X$ be the unique stochastic process solving the stochastic heat equation~\eqref{e:SHE}
started at $0$ from the stationary measure $\eta$, a space white noise on $\T^2$. Assume there exists $C>0$ 
such that $\const\sim \sqrt{C (\log N)^{-1}}$. Then, for any $\varphi\in H^1$, $\{\inonlin_t[X^N](\varphi)\}_{t\geq 0}$ 
converges in distribution to a Brownian motion independent from $B_t(\varphi)\eqdef \int_0^t (-\Delta)^{\half}\xi(\dd s,\varphi)$ 
and its covariance function is the same as in Proposition~\eqref{p:MartConv}. 
\end{corollary}
\begin{proof}
Let $\varphi\in H^1$ and $\pois[\eta](\varphi)$ be the solution of the Poisson equation determined in~\eqref{e:solPoiNL}. 
Notice that the generator of the process $\she$ is $\gensy$ which coincides with $\gens$ once we choose $\donst=1$. 
$\pois(\varphi)$ is a cylinder function and therefore
(for $N$ fixed) belongs to the domain of $\gensy$. By Dynkin's formula we have
\begin{equ}
\pois[X(t)](\varphi)=\pois[X(0)](\varphi)+\int_0^t \gensy \pois[X(s)](\varphi)\dd s+M_t^N(\varphi)
\end{equ}
where the martingale $M_t^N(\varphi)$ has the {\it exact same expression} as $\martnl$ (to see this, apply It\^o's formula 
to $\pois[X(t)](\varphi)$). From the previous we deduce
\begin{equ}
\int_0^t \const\cN^N[X(s)](\varphi)\dd s=\pois[X(0)](\varphi)-\pois[X(t)](\varphi)-M_t^N(\varphi)
\end{equ} 
so that it suffices to study the terms appearing at right hand side. Since the proof of Proposition~\ref{p:MartConv} 
does not depend on the law of $u^N$ as a process, but only on its invariant measure, we conclude that 
$\{M_t^N(\varphi)\}_{t\geq 0}$ converges to a Brownian motion with the covariance prescribed by the statement. 
Concerning the boundary conditions notice that
\begin{equ}
\Exp[\pois[X(t)](\varphi)^2]=\Exp[\pois[X(0)](\varphi)^2]=\E[\pois[\eta](\varphi)^2]
\end{equ}
since $X$ is started at the invariant measure. But now, by Wick's theorem we have 
\begin{equation}\label{e:Pois0}
\E[\pois[\eta](\varphi)^2]\lesssim \const^2\sum_k|k|^2\sum_{\ell+m=k}\frac{1}{(|\ell|^2+|m|^2)^2}|\varphi_{k}|^2\lesssim \const^2\|\varphi\|_{1,2}^2
\end{equation}
which converges to $0$. Collecting the observations made so far, we see that $\{\inonlin_t[X^N](\varphi)\}_{t\geq 0}$ converges 
to a Brownian motion with the covariance prescribed by the statement. 

Independence is a consequence of the fact that for any $N$, $\inonlin_t[X^N](\varphi)$ belongs to the second 
homogeneous Wiener chaos associated to $\xi$ and $\eta$ together. Hence, $\inonlin_t[X^N](\varphi)$ and $\xi$ are uncorrelated, 
and, since the former is bounded in $L^p$, also their covariance converges to $0$. Since the limit of $\inonlin_t[X^N](\varphi)$
is Gaussian and it is uncorrelated from $\xi$, the two are independent. 
\end{proof}

\begin{remark}
The interest in the previous corollary is twofold. First, it provides an example of a situation in which a deterministic
ill-posed operation (in this case the AKPZ nonlinearity $\cN^N$) when suitably rescaled and evaluated at a Gaussian measure, 
produces a {\it new} noise {\it independent} from the one we started with. 
Similar phenomena are observed in situations in which the nonlinearity becomes {\it critical} (in terms of regularity) 
for the equation, and have been observed also in the context of the Isotropic KPZ equation, see~\cite{CSZ, Gu}. 

On a different note, Corollary~\ref{cor:Cherry}, provides a purely probabilistic construction of the first (relevant) stochastic process
one would need to analyse in the context of Regularity Structures~\cite{Hai} and (one of) the reason why the theory 
is not expected to work
if applied to the $2$ dimensional (A)KPZ. Indeed, without entering the details, the approach is based on 
the ability of performing a partial expansion of the solution around the solution of the linearised equation, in which 
the terms appearing can be obtained via a Picard iteration and are increasingly more regular. The expansion is {\it partial} since 
from some point on there is no more gain in regularity and a deterministic argument needs to be invoked in order 
to conclude the fixed point. 

Now, the problem here is that the Picard iteration does not provide any gain in regularity, 
and consequently there is also no point where one could stop. 
Hence, one would end up with an infinite series of stochastic processes
each of which could in principle (as it happens for the so-called ``cherry'') converge to a new white noise, potentially independent 
of the others and there is no hope 
for such a series to be summable!
\end{remark}

\begin{appendix}

\section{An alternative proof of~\eqref{e:simmetry}}\label{sec:alter}

It suffices to notice that the right hand side of~\eqref{e:simmetry} is simply the scalar product of the non-linearity evaluated at 
$\eta$, i.e. $\cN^N(\eta)$, and $\eta$ itself. Setting $\mu\eqdef (-\Delta)^{-\frac{1}{2}}\eta$ we have
\begin{align*}
\sum_{m,l \in \Z_0^2} &\nonlin_{m,l} \eta_m \eta_l \eta_{-m-l}=\langle \cN^N(\eta),\eta\rangle=\langle (-\Delta)( (\partial_1\mu)^2- (\partial_2\mu)^2),\mu\rangle\\
&=\sum_{i=1}^2 \langle \partial_i (\partial_i\mu)^2,\partial_i\mu\rangle 
+ \sum_{\substack{i,j\in\{1,2\}\\i\neq j}}(-1)^i\langle \partial_i(\partial_j\mu)^2,\partial_i\mu\rangle\\
&=\frac{1}{3} \sum_{i=1}^2 \langle \partial_i (\partial_i\mu)^3,1\rangle + 
2 \sum_{\substack{i,j\in\{1,2\}\\i\neq j}}(-1)^i\langle \partial_i\mu\,\partial_j\mu \,\partial_{i,j}\mu),1\rangle
\end{align*}
from which we see that the first sum is $0$ since each summand is, while the second sum vanishes because the two summands 
are the same but they have opposite sign. 

\section{Laplace transform and short time behaviour}\label{ap:HLTT}

In this appendix, we provide a proof of the following Lemma. 

\begin{lemma}\label{lem:ST}
Let $f:\R_+\to\R_+$ be a continuous non negative function. Assume there exist two strictly positive constants, $c< C$, such that
\begin{align}
f(t)&\leq Ct  &\text{for all $t\geq0$}\label{e:UB}\\
\int_0^\infty e^{-\lambda t} f(t)\dd t&\geq \frac{c}{\lambda^2} &\text{for all $\lambda>0$}\label{e:LB}
\end{align}
then,  $f(t)\sim t$ as $t$ converges to $0$. 
\end{lemma}
\begin{proof}
Notice that it suffices to prove that $\limsup_{t\to 0} t^{-1} f(t)>0$. 
We argue by contradiction. Assume $\lim_{t\to 0} t^{-1} f(t)=0$, 
so that for every $\eps>0$ there exists $\delta(\eps)>0$ such that for all $t\leq \delta(\eps)$,
$f(t)<\eps t$. Let $\eps>0$ (to be fixed later) and $a,\lambda>0$ be such that $a/\lambda\leq \delta(\eps)$. Then, 
an easy computation shows that
\begin{equ}
\int_0^{a/\lambda} e^{-\lambda t} f(t)\dd t\leq \eps \int_0^{a/\lambda} e^{-\lambda t} t\dd t 
=\frac{\eps (1-a e^{-a}-e^{-a})}{\lambda^2}\leq \frac{\eps}{\lambda^2}
\end{equ}
while, by~\eqref{e:UB}, we have
\begin{equ}
\int_{a/\lambda}^\infty e^{-\lambda t} f(t)\dd t\leq C \int_{a/\lambda}^\infty e^{-\lambda t} t\dd t=C\frac{(a+1)e^{-a}}{\lambda^2}
\end{equ}
which implies
\begin{equ}
\int_{0}^\infty e^{-\lambda t} f(t)\dd t\leq \frac{\eps+C(a+1)e^{-a} }{\lambda^2}\,.
\end{equ}
But now, if we choose $\eps= c/4$, where $c$ is the constant in~\eqref{e:LB}, and $a$ and $\lambda$ sufficiently large
so that $a/\lambda<\delta(c/4)$ and $C (a+1)e^{-a} <c/4$, then, by~\eqref{e:LB}, we obtain the desired contradiction. 
\end{proof}

\end{appendix}

\bibliography{bibtex}
\bibliographystyle{Martin}

\end{document}